\documentclass[12pt]{article}

\usepackage[a4paper]{geometry}
\usepackage{amssymb}
\usepackage{amsmath}
\usepackage{amsthm}

\usepackage[dvipdfmx]{graphicx}

\theoremstyle{definition}
\newtheorem{theorem}{Theorem}[section]
\newtheorem{proposition}[theorem]{Proposition}
\newtheorem{lemma}[theorem]{Lemma}
\newtheorem{corollary}[theorem]{Corollary}

\newtheorem{definition}[theorem]{Definition}

\newtheorem{assumption}[theorem]{Assumption}
\newtheorem*{definition*}{Definition}
\newtheorem*{example*}{Example}
\newtheorem*{problem*}{Problem}
\newtheorem*{problems*}{Problems}

\theoremstyle{remark}
\newtheorem{remark}[theorem]{Remark}
\newtheorem*{remark*}{Remark}
\newtheorem{remarks}[theorem]{Remarks}
\newtheorem*{remarks*}{Remarks}

\numberwithin{equation}{section}

\title{Characterization of spacetime singularities for the Schr\"odinger equation by initial state}

\author{
Takeru {\scshape Fujii}\footnote{Graduate School of Mathematical Sciences, 
The University of Tokyo, 3-8-1 Komaba, Meguro-ku, Tokyo 153-8914, Japan.
E-mail: \texttt{fujii-takeru942@g.ecc.u-tokyo.ac.jp}. 
}
\ \& 
Kenichi {\scshape Ito}\footnote{Department of Mathematics, Graduate School of Science, Kobe University,
1-1, Rokkodai, Nada-ku, Kobe 657-8501, Japan.
E-mail: \texttt{ito-ken@math.kobe-u.ac.jp}. 
}
}

\date{}

\begin{document}
\allowdisplaybreaks
\maketitle

\begin{abstract}
We discuss spacetime singularities of a solution to the Schr\"odinger equation with a metric perturbation and 
a sublinear potential. 
The \textit{quasi-homogeneous wave front set}, due to Lascar (1977), 
of a solution is characterized by that of the free solution, and a classical high-energy scattering data. 
In the one-dimensional case, it further reduces to the \textit{homogeneous wave front set}, 
due to Nakamura (2005), 
of the initial time-slice. 
For the proof of the former result we implement an idea inspired by Nakamura (2009), 
which was originally devised for \textit{spatial} singularities of the Schr\"odinger equation.  
As for the latter result, we use an exact Egorov-type formula for the free propagator, and a special partition of unity conforming with the classical flow. 
\end{abstract}



\tableofcontents

\section{Settings and results}

\subsection{Introduction}

In this paper, let $\mathbb R^{1+d}=\mathbb R_t\times \mathbb R^d_x$ with $d\in\mathbb N=\{1,2,\ldots\}$, 
and we investigate the singularities of a solution $u\in \mathcal D'(\mathbb R^{1+d})$ to 
the Schr\"odinger equation 
\begin{equation}
\tfrac{\partial}{\partial t}u=-\mathrm iHu,\quad 
u(0,\cdot)=\phi\in \mathcal S'(\mathbb R^d)
.
\label{25012311}
\end{equation}
The Schr\"odinger operator $H$ is time-dependent, and is of the form  
\begin{equation}
 H=\tfrac12 p_i a_{ij}(t,x)p_j+V(t,x) ,
\label{25020420}
\end{equation}
where $ p_i=-\mathrm i\partial/\partial x_i$ for $i=1,\dots,d$, 
and the \textit{Einstein summation convention} is adopted without tensorial superscripts. 
We assume that $H$ is a \textit{short-range} perturbation in the high-energy regime
of the free Schr\"odinger operator 
\begin{equation}
K=-\tfrac12\Delta=\tfrac12 p^2,
\label{260125}
\end{equation}
but precise assumptions will be given later.

Our goal is to characterize \textit{spacetime} singularities 
of a solution $u$ to \eqref{25012311},
as a function of $(t,x)$, 
in terms of the initial state $\phi$. 
So far, concerning singularities of the Schr\"odinger equation, 
the main focus has been on \textit{spatial} singularities, or singularities of time-slices. 
The spacetime singularities have been rarely considered,
except for the earliest works, e.g., by Boutet de Monvel~\cite{MR423430}, Lascar~\cite{MR461592}, 
Parenti--Seg\`ala~\cite{MR623644} and Sakurai~\cite{MR674445,MR798031}, 
but quite recently Gell-Redman--Gomes--Hassell~\cite{MR4995130}
applied their propagation properties to construction of the Poisson operators. 
The spacetime singularities are regaining attention. 

In our first main result, we characterize the \textit{quasi-homogeneous wave front set},
introduced by Lascar~\cite{MR461592}, 
of a perturbed solution by using that of the free solution and the classical high-energy scattering data. 
Lascar~\cite{MR461592} in fact obtained propagation of such wave front set for the Schr\"odinger equation, 
however he could only compare points of the identical time components, 
due to infinite propagation speed. 
We have succeeded in removing such restriction, comparing those of different time components, 
in the sense that the free solution is explicitly written down by the initial state. 
Note that a similar result was obtained for spatial singularities by Nakamura~\cite{MR2488342}, 
and our strategy is strongly motivated by his scattering theoretic approach. 
However, adaptation of the idea to spacetime singularities is non-trivial. 
We cannot anymore use the time variable as a deformation parameter, 
since it is subject to (pseudo)differentiation and integration as one of the base variables.  
The analysis of the associated classical mechanics also gets involved due to
the time-dependence of the Hamiltonian. 

As our second main result, 
we give a sufficient condition for absence of the quasi-homogeneous wave front set, 
in terms of the \textit{homogeneous wave front set}, 
introduced by Nakamura~\cite{MR2115261}, of any other time-slice.
In addition, in the one-dimensional case, it turns out to be also a necessary condition. 
Here, a difficulty is that we have to compare singularities of functions defined on different dimensional spaces.
For the proof we will make use of an exact Egorov-type identity involving the Weyl quantization and the free propagator: 
\begin{equation}
\mathrm e^{\mathrm itK}a^{\mathrm W}(x,p_x)\mathrm e^{-\mathrm itK}
=a^{\mathrm W}(x+tp_x,p_x)
,
\label{25020716}
\end{equation}
see Nakamura's paper~\cite{MR2488342}, 
and a special partition of unity generated by the free classical flow.

While the project was in progress, 
the authors were informed that Gell-Redman--Gomes--Hassell~\cite{MR4995130} had just obtained 
the propagation of spacetime singularities, 
with application to the Poisson operator on the spacetime $\mathbb R^{1+d}$. 
However, note that their perturbations are compactly supported. 
They in fact remarked that their results are extensible 
to non-compactly supported decaying perturbations, 
but our assumption further admits non-decaying sublinear potentials. 
In addition, we emphasize that our techniques are quite different from theirs, 
and would be more elementary and simpler. 
Therefore, this paper would have applications similar to theirs
under more general settings, hopefully, providing additional insight. 

We have also found that former part of our second main result overlaps with 
Szeftel's work~\cite{MR2059146}, where reflection of singularities by an obstacle is discussed. 
However still, our techniques are different and simpler, 
and admit a wider class of perturbations, except for an obstacle. 

Lastly let us briefly review spatial singularities of the Schr\"odinger equation.
There is a considerable amount of literature on this topic, 
beginning from pioneering works by Craig--Kappeler--Strauss \cite{MR1361016}, Yajima~\cite{MR1414302} and Doi~\cite{MR1387689}. 
However, we would like to particularly mention that 
a characterization of the $C^\infty$ spatial wave front set was first settled by Hassell--Wunsch~\cite{MR2178967}, 
and then it was simplified and extended to long-range perturbations by Nakamura~\cite{MR2488342,MR2272875}. 
These characterizations were preceded by 
slightly rough sufficient (and not necessary) conditions by 
Wunsch~\cite{MR1687567} and Nakamura~\cite{MR2115261},
where the \textit{quadratic scattering wave front set} and the \textit{homogeneous wave front set} were introduced, respectively. 
Note that these wave front sets were shown to be equivalent by the second author~\cite{MR2273972}. 
The \textit{global wave front set}, or the \textit{Gabor wave front set}, 
is also well adapted to singularities of the Schr\"odinger equation,
see, e.g., Cordero--Nicola~\cite{MR3249939}, Nicola--Rodino~\cite{MR3415019} and
Cordero--Nicola--Rodino~\cite{MR3317554}, 
and in fact Schulz--Wahlberg~\cite{MR3645728} proved it is equal to the homogeneous wave front set. 
For other technical variants of the wave front set, see, e.g., works by Ito--Nakamura~\cite{MR2567509},  
Fukushima~\cite{fukushima2022propagationsingularitiesschrodingerequations}
and Cappiello--Rodino--Wahlberg~\cite{MR4727203}. 
These variations arise from how we realize the phase space infinities,
but we can also vary categories of smoothness. 
See works by \=Okaji~\cite{MR1869770} and Kato--Kobayashi--Ito~\cite{MR3657227} for the Sobolev singularities, 
Kajitani--Wakabayashi~\cite{MR1766299}, Robbiano--Zuily~\cite{MR1714756,MR1778784,MR1958605}, Takuwa~\cite{MR2082219} and 
Martinez--Nakamura--Sordoni~\cite{MR2237289,MR2554936,MR2763356} for analytic singularities,
and Kajitani--Taglialatela~\cite{MR2003749} and Mizuhara~\cite{MR2498751} for the Gevrey singularities. 
Some of the above works, 
as well as Kato--Ito--Kobayashi~\cite{MR3050803} and Kato--Ito~\cite{MR3309206,MR4343775}, 
employ the wave packet transform, or the short-time Fourier transform, 
instead of pseudodifferential operators. 
For this method we refer to Rodino--Trapasso~\cite{10.1007/978-3-030-61346-4_17}.

\subsection{Assumptions}

First we present precise assumptions of the paper. 
We let $\mathbb N_0=\{0\}\cup\mathbb N$
and $\langle x\rangle=(1+x^2)^{1/2}$, and 
denote by $\delta=(\delta_{ij})_{i,j=1,\dots,d}$ the identity matrix. 

\begin{assumption}\label{250124}
Let $a=(a_{ij})_{i,j=1,\dots,d}\in C^\infty(\mathbb R^{1+d};\mathbb R^{d\times d})$ and $V\in C^\infty(\mathbb R^{1+d};\mathbb R)$ 
satisfy the following.
\begin{enumerate}
\item
For each $(t,x)\in\mathbb R^{1+d}$ the matrix $a(t,x)$ is symmetric and positive definite. 
\item
There exists $\epsilon>0$ such that 
for any $\alpha=(\alpha_0,\alpha')\in\mathbb N_0\times \mathbb N_0^d$ 
there exists $C>0$ such that for any $i,j=1,\dots,d$ and $(t,x)\in \mathbb R^{1+d}$
\[
|\partial^{\alpha} (a_{ij}(t,x)-\delta_{ij})|\le C\langle x\rangle^{-1-|\alpha'|-\epsilon}
,\quad 
|\partial^{\alpha} V(t,x)|\le C\langle x\rangle^{1-|\alpha'|-\epsilon}
.
\]
\end{enumerate}
\end{assumption}
\begin{remark}
These are the so-called short-range conditions in the high-energy regime. 
In fact, every \textit{non-trapped} classical orbit approaches a free one in the high-energy limit, 
see Definition~\ref{25080316} and Proposition~\ref{250123}. 
Note that for the purpose of the paper it suffices to assume the above estimates locally in time, 
however we let them be global for simplicity.
\end{remark}

Under Assumption~\ref{250124} the Cauchy problem \eqref{25012311} is well-posed in the following sense. 
We quote a statement from Yajima's paper~\cite{MR1243098}. 

\begin{theorem}[\cite{MR1243098}]\label{260124}
Suppose Assumption~\ref{250124}. For any $k\in\mathbb N_0$ set 
\begin{equation*}
\Sigma(k)=\bigl\{\phi \in L^2(\mathbb R^d);\ x^\alpha\partial^\beta\phi\in L^2(\mathbb R^d)
\ \text{for all }\alpha,\beta\in\mathbb N_0^d \ \text{with }|\alpha|+|\beta|\le k
\bigr\}
,
\end{equation*}
and denote its dual space by $\Sigma(-k)=\Sigma(k)^*$. 
Then there exists a unique family $\{U(t,s)\}_{t,s\in\mathbb R}$ of isomorphisms on $\mathcal S'(\mathbb R^d)$ 
such that the following holds. 
\begin{enumerate}
\item
For any $s,t\in\mathbb R$ and $k\in\mathbb Z$, $U(s,t)$ restricts to an isomorphism on $\Sigma(k)$, 
and to a unitary operator on $\Sigma(0)=L^2(\mathbb R^d)$.
\item
For any $t,s,r\in\mathbb R$, $U(t,s)U(s,r)=U(t,r)$.
\item
For any $k\in\mathbb Z$ the mapping $\mathbb R^2\to\mathcal L(\Sigma(k)),\ (t,s)\mapsto U(t,s)$ is strongly continuous. 
\item
For any $k\in\mathbb Z$ the mapping 
$\mathbb R^2\to\Sigma(k-2),\  (t,s)\mapsto U(t,s)\phi$ is strongly continuously differentiable with 
strong partial derivatives 
\[
\tfrac{\partial}{\partial t}U(t,s)=-\mathrm iHU(t,s),\quad 
\tfrac{\partial}{\partial s}U(t,s)=\mathrm iU(t,s)H
.
\]
\end{enumerate}
\end{theorem}

\begin{remark}
By Theorem~\ref{260124} we can solve \eqref{25012311} for general $\phi\in\mathcal S'(\mathbb R^d)$ 
since $\bigcup_{k\in\mathbb Z}\Sigma(k)=\mathcal S'(\mathbb R^d)$. 
\end{remark}

We always fix the initial time at $s=0$, and thus abbreviate $U(t)=U(t,0)$. 
Now the purpose of the paper is to identify the singularities of 
a solution $u\in \mathcal D'(\mathbb R^{1+d})$ to \eqref{25012311} given by 
\begin{equation}
u(t,x)=(U(t)\phi)(x)
.
\label{25012917}
\end{equation}
A characterization will be given in terms of the free solution 
$u_K\in \mathcal S'(\mathbb R^{1+d})$ defined as 
\begin{equation}
u_K(t,x)=(\mathrm e^{-\mathrm itK}\phi)(x)
,
\label{250204}
\end{equation}
see \eqref{260125} for $K$. 
In the one-dimensional case, we will further rewrite it 
more directly by the initial time-slice $\phi\in\mathcal S'(\mathbb R^d)$.

\subsection{Quasi-homogeneous wave front set}

In the analysis of singularities for a Schr\"odinger-type equation, 
it is more natural to discuss the \textit{quasi-homogeneous wave front set} introduced by Lascar~\cite{MR461592}. 
Let us reformulate it in the semiclassical manner. 
Recall that, in a general dimensional space $\mathbb R^n$ with $n\in\mathbb N$, 
the \textit{Weyl quantization} of a symbol $a\in C^\infty_{\mathrm c}(\mathbb R^{2n})$ 
is defined formally for any $v\in\mathcal S'(\mathbb R^{n})$ as 
\begin{align*}
&a^{\mathrm W}(z,p_z)v(z)
=(2\pi)^{-n}\int_{\mathbb R^{2n}}\mathrm e^{\mathrm i(z-w)\zeta}
a\bigl(\tfrac{z+w}2,\zeta\bigr)v(w)\,\mathrm dw\mathrm d\zeta
.
\end{align*}

\begin{definition}
Define the \textit{quasi-homogeneous wave front set of order} $\theta\in (0,\infty)$ of $v\in \mathcal S'(\mathbb R^{1+d})$: 
\[\mathop{\mathrm{qh\text{-}WF}}\nolimits^\theta (v)\subset \mathbb R^{1+d}\times (\mathbb R^{1+d}\setminus\{0\})\] 
as a complement of the set of all $(s,y,\sigma,\eta)\in\mathbb R^{1+d}\times (\mathbb R^{1+d}\setminus\{0\})$ such that 
there exists $a\in C^\infty_{\mathrm c}(\mathbb R^{2(1+d)})$  satisfying 
\[
a(s,y,\sigma,\eta)\neq 0,\qquad 
 \|a^{\mathrm W}(t,x,h^\theta p_t,hp_x)v\|_{L^2_{t,x}}=\mathcal O(h^\infty)\ \ \text{as }h\to +0
.
\]
\end{definition}

\begin{remarks}\label{250207}
\begin{enumerate}
\item
$\mathop{\mathrm{qh\text{-}WF}}\nolimits^1 (v)$
coincides with the ordinary wave front set $\mathop{\mathrm{WF}}(v)$. 

\item
In general, for any $0<\rho<\theta<\infty$,  
$\mathop{\mathrm{qh\text{-}WF}}\nolimits^{\theta} (v)$ refines 
$\mathop{\mathrm{qh\text{-}WF}}\nolimits^\rho (v)$ at $\{\xi=0\}$, 
while it degrades the rest down to $\{\tau=0\}$. 
If we borrow terminologies from Melrose~\cite{MR1291640}, $\mathop{\mathrm{qh\text{-}WF}}\nolimits^{\theta} (v)$ is simultaneously a \textit{blow-up} and a \textit{blow-down} 
of $\mathop{\mathrm{qh\text{-}WF}}\nolimits^\rho (v)$ at $\{\xi=0\}$ and $\{\tau=0\}$, respectively. 
\end{enumerate}
\end{remarks}

The Schr\"odinger equation from \eqref{25012311} is apparently ``quasi-homogeneous'' in the $(t,x)$-derivatives, 
or in the $(\tau,\xi)$ variables.
In fact, according to the following inclusion relations, we should choose $\theta=2$. 
\begin{proposition}[\cite{MR461592}]\label{25012918}
Suppose Assumption~\ref{250124}, and let $u\in\mathcal D'(\mathbb R^{1+d})$ be a solution to \eqref{25012311} 
given by \eqref{25012917}. 
Then one has 
\[
\mathop{\mathrm{qh\text{-}WF}}\nolimits^\theta (u)
\subset
\begin{cases} 
\{(t,x,\tau,0);\ \tau\neq 0\}
& \text{if }\theta\in (0,2),\\ 
\bigl\{\bigl(t,x,-\tfrac12a_{ij}(t,x)\xi_i\xi_j,\xi\bigr);\ \xi\neq 0\bigr\}&\text{if }\theta=2,\\ 
\{(t,x,0,\xi);\ \xi\neq 0\}
&\text{if }\theta\in (2,\infty).
\end{cases}
\]
\end{proposition}
\begin{remark}
We can define $\mathop{\mathrm{qh\text{-}WF}}\nolimits^\theta (u)$ 
for $u$ from \eqref{25012917} 
even if $u\notin \mathcal S'(\mathbb R^{1+d})$. 
In fact, we can modify $u$ to be in $\mathcal S'(\mathbb R^{1+d})$ without changing the values on any compact time-interval
by smoothly deforming $H$ to $K$ outside a larger time-interval. 
\end{remark}

The proof of Proposition~\ref{25012918} is straightforward; 
Given $\theta>0$, define an appropriate \textit{microlocal ellipticity} for the quantization 
$a^{\mathrm W}(t,x,h^\theta p_t,hp_x)$,
and repeat the standard parametrix construction. 
For the details see Lascar's paper~\cite{MR461592}.

\subsection{Classical high-energy scattering data}

To state the first main theorem, we need the classical high-energy scattering data. 
By the short-range nature of our perturbations, 
we can at last remove the time-dependence and the potential 
from the classical mechanics corresponding to \eqref{25020420}. 
As a result, we may replace the high-energy limit by the large time limit given below. 
We postpone such reduction procedure to Section~\ref{2601251432},
and here only formulate the scattering data in the reduced setting.

Let us fix $s\in\mathbb R$, and consider a time-independent classical Hamiltonian  
\[
H_s(x,\xi)=\tfrac12a_{ij}(s,x)\xi_i\xi_j,\quad (x,\xi)\in\mathbb R^{2d}. 
\]
Denote by 
\begin{equation}
(x(t),\xi(t))=(x(t;s,y,\eta),\xi(t;s,y,\eta))
\label{250620}
\end{equation}
a solution to the associated Hamilton equations
\begin{equation}
\dot x_i=a_{ij}(s,x)\xi_j,\ \  
\dot\xi_i=-\tfrac12(\partial_ia_{jk}(s,x))\xi_j\xi_k
\quad \text{for }i=1,\dots,d
\label{26030522}
\end{equation}
with initial data $(x(s),\xi(s))=(y,\eta)$. 

\begin{definition}\label{25080316}
A point $(s,y,\eta)\in \mathbb R^{1+2d}$ is said to be \textit{forward}/\textit{backward non-trapping} if  
\[
\lim_{t\to\pm\infty} |x(t;s,y,\eta)|=\infty
,
\]
respectively. 
In addition, the sets of all forward/backward non-trapping points are denoted by $\Omega_\pm\subset \mathbb R^{1+2d}$, respectively. 
\end{definition}

\begin{remark}
By definition it is clear that $\Omega_\pm\subset \mathbb R^{1+d}\times(\mathbb R^d\setminus\{0\})$.
\end{remark}

We quote the following result from Nakamura~\cite{MR2488342} without proof. 

\begin{proposition}[\cite{MR2488342}]\label{250123}
Suppose Assumption~\ref{250124}.
Then $\Omega_\pm\subset \mathbb R^{1+2d}$ are open, and there exist the limits 
\begin{align*}
x_\pm&=x_\pm(s,y,\eta):=\lim_{t\to\pm\infty}\bigl(x(t;s,y,\eta)- (t-s)\xi(t;s,y,\eta)\bigr),
\\
\xi_\pm&=\xi_\pm(s,y,\eta):=\lim_{t\to\pm \infty}\xi(t;s,y,\eta)
\end{align*}
locally uniformly in $(s,y,\eta)\in \Omega_\pm$, respectively. 
\end{proposition}


\subsection{The first main result}

Now, we present the first main theorem of the paper. 

\begin{theorem}\label{250208}
Suppose Assumption~\ref{250124}. For any $\phi\in\mathcal S'(\mathbb R^d)$ 
let $u\in\mathcal D'(\mathbb R^{1+d})$ and $u_K\in \mathcal S'(\mathbb R^{1+d})$ be from 
\eqref{25012917} and \eqref{250204}, respectively. 
In addition, for any $(s,y,\eta)\in \Omega_\pm$ with $\pm s<0$ 
let $(x_\pm,\xi_\pm)$ be from Proposition~\ref{250123}, respectively. 
Then one has 
\begin{align*}
&\bigl(s,y, -\tfrac12a_{ij}(s,y)\eta_i\eta_j,\eta\bigr)\in \mathop{\mathrm{qh\text{-}WF}}\nolimits^2(u)
\\&\text{if and only if}\ \ 
\bigl(s,x_\pm,-\tfrac12\xi_\pm^2,\xi_\pm\bigr)\in\mathop{\mathrm{qh\text{-}WF}}\nolimits^2(u_K),
\end{align*}
respectively. 
\end{theorem}
\begin{remarks}
\begin{enumerate}
\item
By Proposition~\ref{25012918}, for $\mathop{\mathrm{qh\text{-}WF}}\nolimits^2(u)$, 
it suffices to consider the points of the form 
$(s,y, -\tfrac12a_{ij}(s,y)\eta_i\eta_j,\eta)$. 
\item 
We may understand $\mathop{\mathrm{qh\text{-}WF}}\nolimits^2(u_K)$ as written in terms of $\phi$,
since the free solution $u_K$ has an explicit representation involving $\phi$. 
Thus Theorem~\ref{250208} compares (phase space) singularities of $u$ of different times. 
See Lascar~\cite{MR461592}, and also Szeftel~\cite{MR2059146}, 
for the results comparing singularities of the same time. 
\end{enumerate}
\end{remarks}

The proof of Theorem~\ref{250208} will be given in Section~\ref{26012514}. 
Our basic strategy is directly inspired by Nakamura~\cite{MR2488342}, 
which discussed spatial singularities of time-slices of $u$. 
To see microlocal correspondence between $u$ and $u_K$, 
we shall introduce appropriate deformation, 
and keep track of support of a symbol defining the quasi-homogeneous wave front set. 
Then the symbol is required to satisfy a certain Heisenberg equation, 
and we are to solve it by asymptotic construction. Thus the proof reduces to analysis of the associated classical flow. 
Note that, unlike Nakamura~\cite{MR2488342}, we cannot simply take $t$ as a deformation parameter 
since pseudodifferential operators act on it, and it is also subject to integrations. 
In addition, the analysis of the classical mechanics gets more demanding 
due to the time-dependence of perturbation.

\subsection{Homogeneous wave front set}

Before the second main theorem of the paper, here we recall another variant of the wave front set, 
the \textit{homogeneous wave front set} introduced by Nakamura~\cite{MR2115261}. 

\begin{definition}
Define the \textit{homogeneous wave front set} of $\phi\in \mathcal S'(\mathbb R^d)$:  
\[
\mathop{\mathrm{HWF}}\nolimits(\phi)\subset \mathbb R^{2d}\setminus\{0\}
\]
as a complement of the set of all $(y,\eta)\in \mathbb R^{2d}\setminus\{0\}$ such that 
there exists $a\in C^\infty_{\mathrm c}(\mathbb R^{2d})$ 
satisfying 
\[a(y,\eta)\neq 0,\qquad 
\|a^{\mathrm W}(hx,hp_x)\phi\|_{L^2_x}=\mathcal O(h^\infty)\quad \text{as }h\to +0
.
\]

\end{definition}

\begin{remarks}
\begin{enumerate}
\item
Obviously, the homogeneous wave front set simultaneously measures singularity and growth at infinity of a function on $\mathbb R^d$. 
The second author~\cite{MR2273972} and Schulz--Wahlberg~\cite{MR3645728} 
proved that it is essentially equivalent to the 
\textit{quadratic scattering wave front set} due to Wunsch~\cite{MR1687567} 
and to the \textit{global} (\textit{Gabor}) \textit{wave front set} due to H\"ormander~\cite{MR1178557},
respectively.
\item
The quasi-homogeneous and the homogeneous wave front sets 
have the adjective ``homogeneous'' in different ways. 
The former refers to the homogeneity within the Fourier variables,
while the latter to that in the configuration and the Fourier variables. 
\end{enumerate}
\end{remarks}


\subsection{The second main result}

Finally, we present the second main theorem and a corollary, 
which in the one-dimensional case provide a necessary and sufficient condition. 

\begin{theorem}\label{250208b}
For any $\phi\in \mathcal S'(\mathbb R^d)$ let $u_K\in \mathcal S'(\mathbb R^{1+d})$ be from \eqref{250204}. 
Then for any $(s,y,\eta)\in \mathbb R^{1+d}\times(\mathbb R^d\setminus\{0\})$
\[
\bigl(s,y,-\tfrac12\eta^2,\eta\bigr)\in\mathop{\mathrm{qh\text{-}WF}}\nolimits^2(u_K)
\ \ \text{implies}\ \  
(-s\eta,\eta)\in \mathop{\mathrm{HWF}}(\phi)
.
\]
Moreover, the converse is true if $d=1$. 
\end{theorem}

\begin{remarks}
\begin{enumerate}
\item
Similarly to Theorem~\ref{250208}, for $\mathop{\mathrm{qh\text{-}WF}}\nolimits^2(u_K)$ 
it suffices to discuss the points of the form $(s,y,-\tfrac12\eta^2,\eta)$ by Proposition~\ref{25012918}. 
\item
See Szeftel~\cite[Corollaire 4.2]{MR2059146} for a result similar to the former assertion,
and also Nakamura~\cite{MR2115261} for its spatial version. 
Note that the latter assertion is not covered by Szeftel~\cite{MR2059146}, 
and is never true for spatial singularities. 
\end{enumerate}
\end{remarks}

The proof of Theorem~\ref{250208b} will be given in Section~\ref{26012420}. 
Even though $u_K$ has an explicit integral expression in terms of $\phi$, 
it is not so straightforward as it seems  
since $u_K$ and $\phi$ live in different dimensional spaces. 
In particular, for the converse part we have to construct a special partition of unity
that conforms with the free classical flow on the phase space. 
There the identity \eqref{25020716} is very useful, and we will repeatedly use it.

By combining the results so far, in the one-dimensional space, 
we can characterize the quasi-homogeneous wave front set of $u$ by one of time-slices $U(r)\phi$. 

\begin{corollary}\label{25020815}
Under the settings of Theorem~\ref{250208} with $d=1$, let $\pm r>\pm s$, respectively. 
Then one has 
\begin{align*}
&\bigl(s,y, -\tfrac12a_{ij}(s,y)\eta_i\eta_j,\eta\bigr)\in \mathop{\mathrm{qh\text{-}WF}}\nolimits^2(u)
\\&
\text{if and only if}\ \ 
((r-s)\xi_\pm,\xi_\pm)\in \mathop{\mathrm{HWF}}\nolimits(U(r)\phi), 
\end{align*}
respectively.

\end{corollary}
\begin{proof}
The assertion is straightforward from Theorems~\ref{250208} and \ref{250208b}. 
\end{proof}

\section{Classical mechanics in high-energy regime}\label{2601251432}

In this section we study the classical mechanics that will be needed in Section~\ref{26012514}. 
We introduce a certain technical classical Hamiltonian corresponding 
to an operator appearing in Section~\ref{26012514}, 
and investigate the high-energy limit of its Hamiltonian flow. 
In spite of its complicated appearance,  after appropriate changes of variables, 
it reduces to typical classical Hamiltonians, and the standard methods work well. 

In Section~\ref{26030320} we state the main proposition of the section. 
Section~\ref{26030315} is devoted to preliminaries for the proof,
and finally in Section~\ref{260303} we implement the proof.

\subsection{Classical Hamiltonian flow}\label{26030320}

Here we present the settings and the main result of the section. 
We consider the classical Hamiltonian 
\begin{align}
\begin{split}
l_0(\kappa,t,x,\tau,\xi)&=
-\tfrac12 t \bigl\{a_{ij}((1-\kappa)t,x-\kappa t\xi)-\delta_{ij}\bigr\}\xi_i\xi_j
\\&\phantom{{}={}}{}
-tV((1-\kappa)t,x-\kappa t\xi),
\end{split}
\label{250619}
\end{align}
which is formally a principal symbol of a technical operator \eqref{2603014} from Section~\ref{26012514}. 
For motivation see the arguments there. 
We note that for our purpose we may drop the last term on the right-hand side of \eqref{250619},
but we have decided to just keep it. 
Let us study the high-energy limit of the associated Hamiltonian flow. 
To be more precise, consider the associated Hamilton equations 
\begin{align}
\tfrac{\mathrm d}{\mathrm d\kappa}t&=0
,
\label{25042320}
\\ 
\begin{split}
\tfrac{\mathrm d}{\mathrm d\kappa} x_i&
=
-t \bigl\{a_{ij}((1-\kappa)t,x-\kappa t\xi)-\delta_{ij}\bigr\}\xi_j
\\&\quad{}
+\tfrac12 \kappa t^2 (\partial_ia_{jk})((1-\kappa)t,x-\kappa t\xi)\xi_j\xi_k
+\kappa t^2(\partial_iV)((1-\kappa)t,x-\kappa t\xi)
,
\end{split}
\label{25042321}
\\
\begin{split}
\tfrac{\mathrm d}{\mathrm d\kappa}\tau&
=
\tfrac12  \bigl\{a_{ij}((1-\kappa)t,x-\kappa t\xi)-\delta_{ij}\bigr\}\xi_i\xi_j
\\&\quad{}
+\tfrac12 (1-\kappa)t(\partial_ta_{ij})((1-\kappa)t,x-\kappa t\xi)\xi_i\xi_j
\\&\quad{}
-\tfrac12 \kappa t\xi_k(\partial_ka_{ij})((1-\kappa)t,x-\kappa t\xi)\xi_i\xi_j
+V((1-\kappa)t,x-\kappa t\xi)
\\&\quad{}
+(1-\kappa)t(\partial_tV)((1-\kappa)t,x-\kappa t\xi)
-\kappa t\xi_i(\partial_iV)((1-\kappa)t,x-\kappa t\xi)
,
\end{split}
\label{25042322}
\\ 
\tfrac{\mathrm d}{\mathrm d\kappa}\xi_i
&=
\tfrac12 t (\partial_ia_{jk})((1-\kappa)t,x-\kappa t\xi)\xi_j\xi_k
+t(\partial_iV)((1-\kappa)t,x-\kappa t\xi)
\label{25042323}
\end{align}
with initial data 
\begin{equation}
(t(0),x(0),\tau(0),\xi(0))=(s,y,\sigma,\eta). 
\label{25042324}
\end{equation}
Throughout the paper we restrict parameter $\kappa$ to an interval $[0,1]$. 
We then denote the maximally defined Hamiltonian flow generated by \eqref{25042320}--\eqref{25042324} 
by 
\[
\Phi\colon \mathcal U\to \mathbb R^{2(1+d)}
\] 
with $\mathcal U\subset [0,1]\times \mathbb R^{2(1+d)}$ being an open subset
containing $\{0\}\times\mathbb R^{2(1+d)}$. 
In addition, for any $h\in (0,1]$ we introduce a scaling transformation 
\[
\Theta_h\colon \mathbb R^{2(1+d)}\to\mathbb R^{2(1+d)},\ \ (t,x,\tau,\xi)\mapsto (t,x,h^2\tau,h\xi)
,
\]
and set 
\begin{equation}
\mathcal U_h=\Theta_h(\mathcal U)
,\quad 
\Phi_h=\Theta_h\circ\Phi\circ\Theta_{1/h}
\colon \mathcal U_h\to \mathbb R^{2(1+d)}
,
\label{2603015}
\end{equation}
where we have identified $\Theta_h$ with $\mathop{\mathrm{id}}_{[0,1]}\times \Theta_h$. 
Now we state the main proposition of the section. 
Recall notation from Definition~\ref{25080316}.

\begin{proposition}\label{2603011619}
\begin{enumerate}
\item
For any $\kappa\in[0,1]$ and $h\in(0,1]$ let  
\begin{align*}
\Phi_h(\kappa)
&=\Phi_h(\kappa;{}\cdot{},{}\cdot{},{}\cdot{},{}\cdot{})
,\\
\mathcal U_h(\kappa)
&=\bigl\{(s,y,\sigma,\eta)\in \mathbb R^{2(1+d)};\ (\kappa,s,y,\sigma,\eta)\in \mathcal U_h\bigr\}
.
\end{align*}
Then $\Phi_h(\kappa)$ is a diffeomorphism from $\mathcal U_h(\kappa)$ to its image. 
\item 
For any $(s_0,y_0,\eta_0)\in \Omega_\pm$ with $\pm s_0<0$ there exist
$h_0\in (0,1]$ and  
neighborhoods $U_\pm\subset \mathbb R^{2(1+d)}$ of $(s_0,y_0,-\tfrac12a_{ij}(s_0,y_0)\eta_{0,i}\eta_{0,j},\eta_0)$ 
such that for any $h\in (0,h_0]$ 
\begin{equation}
[0,1]\times U_\pm\subset \mathcal U_h
,
\label{260305}
\end{equation}
and that uniformly in $(s,y,\sigma,\eta)\in U_\pm$
\begin{align}
\begin{split}
&\lim_{\kappa/h\to\infty}\Phi_h(\kappa,s,y,\sigma,\eta)
\\&
=
\bigl(s,x_\pm(s,y,\eta),\sigma+\tfrac12a_{ij}(s,y)\eta_i\eta_j-\tfrac12\xi_\pm(s,y,\eta)^2,\xi_\pm(s,y,\eta)\bigr)
,
\end{split}
\label{2603082229}
\end{align}
respectively. 
\item
Under the setting of the assertion 2, furthermore, for any $\alpha\in \mathbb N_0^{2(1+d)}$ 
there exists $C>0$ such that uniformly in $h\in (0,h_0]$ and 
$(\kappa,s,y,\sigma,\eta)\in [0,1]\times U_\pm$
\[
\bigl|\partial_{s,y,\sigma,\eta}^\alpha \Phi_h(\kappa,s,y,\sigma,\eta)\bigr|
\le C
.
\]
\end{enumerate}
\end{proposition}
\begin{remark}\label{26030822}
We can prove the inclusion relation \eqref{260305} by using the basic unique existence theorem 
for a solution to ODE combined with the following two facts, cf.\ the proof of Lemma~\ref{250804}: 
As $h\to+0$, $\Phi_h$ approaches the flow given by \eqref{250620} for any $\kappa$ of order $\mathcal O(h)$; 
Outside a large compact subset outgoing/incoming orbits almost conserve the energies. 
The arguments are elementary but a bit involved, and we would like to omit them. 
Below we let $\Phi_h$ exist where we are considering. 
\end{remark}

\subsection{Reduction to ordinary Hamiltonians}\label{26030315}

\subsubsection{Time-dependent Hamiltonian}

Our interest is in the slightly more general flow $\Phi_h$ from \eqref{2603015} rather than $\Phi$ itself. 
Here we deduce rescaled equations for $\Phi_h$, 
and investigate their properties, instead of the original ones \eqref{25042320}--\eqref{25042323}. 
First, it is straightforward to see that the components of $\Phi_h$ satisfy the Hamilton equations 
\begin{align}
\tfrac{\mathrm d}{\mathrm d\kappa}t&=0
,
\label{25042320b}
\\ 
\begin{split}
\tfrac{\mathrm d}{\mathrm d\kappa} x_i&
=
-h^{-1}t \bigl\{a_{ij}((1-\kappa)t,x-\kappa th^{-1}\xi)-\delta_{ij}\bigr\}\xi_j
\\&\quad{}
+\tfrac12 h^{-2}\kappa t^2 (\partial_ia_{jk})((1-\kappa)t,x-\kappa th^{-1}\xi)\xi_j\xi_k
\\&\quad{}
+\kappa t^2(\partial_iV)((1-\kappa)t,x-\kappa th^{-1}\xi)
,
\end{split}
\label{25042321b}
\\
\begin{split}
\tfrac{\mathrm d}{\mathrm d\kappa}\tau&
=
\tfrac12  \bigl\{a_{ij}((1-\kappa)t,x-\kappa th^{-1}\xi)-\delta_{ij}\bigr\}\xi_i\xi_j
\\&\quad{}
+\tfrac12(1-\kappa)t(\partial_ta_{ij})((1-\kappa)t,x-\kappa th^{-1}\xi)\xi_i\xi_j
\\&\quad{}
-\tfrac12h^{-1}\kappa t\xi_k(\partial_ka_{ij})((1-\kappa)t,x-\kappa th^{-1}\xi)\xi_i\xi_j
\\&\quad{}
+h^2V((1-\kappa)t,x-\kappa th^{-1}\xi)
\\&\quad{}
+h^2(1-\kappa)t(\partial_tV)((1-\kappa)t,x-\kappa th^{-1}\xi)
\\&\quad{}
-h\kappa t\xi_i(\partial_iV)((1-\kappa)t,x-\kappa th^{-1}\xi)
,
\end{split}
\label{25042322b}
\\ 
\begin{split}
\tfrac{\mathrm d}{\mathrm d\kappa}\xi_i
&=
\tfrac12 h^{-1}t (\partial_ia_{jk})((1-\kappa)t,x-\kappa th^{-1}\xi)\xi_j\xi_k
\\&\quad{}
+ht(\partial_iV)((1-\kappa)t,x-\kappa th^{-1}\xi)
\end{split}
\label{25042323b}
\end{align}
with initial data 
\begin{equation}
(t(0),x(0),\tau(0),\xi(0))=(s,y,\sigma,\eta). 
\label{25042324b}
\end{equation}
These rescaled equations \eqref{25042320b}--\eqref{25042323b} seem quite complicated too, 
but we can rewrite them into a simpler form by changes of variables.
In fact, by \eqref{25042320b} and \eqref{25042324b} it follows that 
\begin{equation}
t\equiv s.
\label{2603082230}
\end{equation}
Substitute it to \eqref{25042321b}--\eqref{25042323b}, 
change the independent variable $\kappa$ to $\mu=(1-\kappa)s$, 
and then the dependent variables $x,\tau,\xi$ to  
\begin{align}
z&=x-h^{-1}(s-\mu)\xi,
\label{2603082231}
\\
\begin{split}
\rho&=\tau
+\tfrac{\mu}s\bigl\{
\tfrac12 a_{ij}(\mu,x-h^{-1}(s-\mu)\xi)\xi_i\xi_j+h^2V(\mu,x-h^{-1}(s-\mu)\xi)\bigr\}
\\&\phantom{{}={}}{}
+\tfrac{s-\mu}{s}\tfrac12\xi^2
,
\end{split}
\label{2603082232}
\\
\zeta&=\xi,
\label{2603082233}
\end{align}
respectively. Thus we obtain 
\begin{align}
\tfrac{\mathrm d }{\mathrm d\mu}z_i
&=
h^{-1}a_{ij}(\mu,z)\zeta_j
,
\label{25061917}
\\
\tfrac{\mathrm d }{\mathrm d\mu}\rho
&=
0
,
\label{25061919}
\\
\tfrac{\mathrm d}{\mathrm d\mu}\zeta_i
&=
-\tfrac12 h^{-1}(\partial_ia_{jk})(\mu,z)\zeta_j\zeta_k
-h(\partial_iV)(\mu,z)
\label{25061918}
\end{align}
with initial data 
\begin{equation}
(z(s),\rho(s),\zeta(s))
=\bigl(y,\sigma+\tfrac12 a_{ij}(s,y)\eta_i\eta_j+h^2V(s,y),\eta\bigr).
\label{2506191919}
\end{equation}
Again, the equation \eqref{25061919} for $\rho$ with \eqref{2506191919} is trivially solved as 
\begin{equation}
\rho\equiv \sigma+\tfrac12 a_{ij}(s,y)\eta_i\eta_j+h^2V(s,y) . 
\label{2603082234}
\end{equation}
Thus it suffices to investigate the equations \eqref{25061917} and \eqref{25061918} for $(z,\zeta)$
with \eqref{2506191919}.
They are exactly the Hamilton equations for the time-dependent Hamiltonian
\[
H(\mu,z,\zeta)=\tfrac12 h^{-1} a_{ij}(\mu,z)\zeta_i\zeta_j+hV(\mu,z)
.
\]
Let us denote a solution to \eqref{25061917}, \eqref{25061918} and \eqref{2506191919} by 
\[
(z(\mu),\zeta(\mu))=(z(\mu;s,y,\eta),\zeta(\mu;s,y,\eta))
.
\]
The \textit{classical Mourre-type estimate} is essential in the following estimates.

\begin{lemma}\label{250804}
For any $(s,y,\eta)\in \Omega_\pm$ with $\pm s<0$ there exist $h_0\in (0,1]$ and $C>0$ such that 
for any $h\in (0,h_0]$, $(s,y,\eta)\in U_\pm$ and $\pm \mu\in [\pm s,0]$ 
\[
\bigl||z(\mu;s,y,\eta)|
\mp h^{-1}(\mu-s)(a_{ij}(s,y)\eta_{i}\eta_{j})^{1/2}
\bigr|
\le C
,
\]
respectively. Moreover, $C>0$ and $h_0\in (0,1]$ can be chosen locally uniformly in $(s,y,\eta)\in \Omega_\pm$.
\end{lemma}
\begin{proof}
We discuss only the upper sign, since the lower one can be treated similarly. 
We also note that all the following estimates are locally uniform in $(s,y,\eta)\in \Omega_\pm$, 
so that the last assertion is automatically proved without explicitly mentioned. 

\smallskip
\noindent
\textit{Step 1.}\ 
We first deduce a rough kinetic energy estimate. 
Let us differentiate 
\begin{align}
\begin{split}
\tfrac{\mathrm d}{\mathrm d\mu}\bigl(a_{ij}(\mu,z)\zeta_{i}\zeta_{j}\bigr)
&
=
(\partial_ta_{ij})(\mu,z)\zeta_{i}\zeta_{j}
+h^{-1}  (\partial_ka_{ij})(\mu,z)a_{kl}(\mu,z)\zeta_{l}\zeta_{i}\zeta_{j}
\\&\phantom{{}={}}{}
+2a_{ij}(\mu,z)\bigl\{-\tfrac12 h^{-1}(\partial_ia_{kl}(\mu,z))\zeta_{k}\zeta_{l}
-h(\partial_iV)(\mu,z)\bigr\}\zeta_{j}
\\&
=
(\partial_ta_{ij})(\mu,z)\zeta_{i}\zeta_{j}
-2ha_{ij}(\mu,z)(\partial_iV)(\mu,z)\zeta_{j}
.
\end{split}
\label{25062016}
\end{align}
By the Cauchy--Schwarz inequality this implies 
\[
\bigl|\tfrac{\mathrm d}{\mathrm d\mu}\bigl(a_{ij}(\mu,z)\zeta_{i}\zeta_{j}\bigr)\bigr|
\le 
C_1a_{ij}(\mu,z)\zeta_{i}\zeta_{j}
+C_1h^2,
\]
so that uniformly in $\mu\in [s,0]$
\[
\mathrm e^{-C_1(\mu-s)}a_{ij}(s,y)\eta_{i}\eta_{j}
-h^2
\le 
a_{ij}(\mu,z)\zeta_{i}\zeta_{j}
\le  
\mathrm e^{C_1(\mu-s)}a_{ij}(s,y)\eta_{i}\eta_{j}
+h^2
.
\]
Hence by letting $h_0\in (0,1]$ be small enough 
it follows that for any $h\in (0,h_0]$ and $\mu\in [s,0]$
\begin{equation}
0<c_1\le a_{ij}(\mu,z)\zeta_{i}\zeta_{j} \le C_2<\infty.
\label{25061920}
\end{equation}
We remark that due to \eqref{25061917} and \eqref{25061920} we in particular have 
for any $h\in(0,h_0]$ and $\mu\in [s,0]$
\begin{equation}
|z(\mu;s,y,\eta)|\le C_3h^{-1}(\mu-s)+|y|
\label{2506192030}
.
\end{equation}

\smallskip
\noindent
\textit{Step 2.}\ 
We next deduce the classical Mourre-type estimate. 
We differentiate 
\begin{align}
\begin{split}
\tfrac{\mathrm d^2}{\mathrm d\mu^2}z^2
&=2h^{-1}\tfrac{\mathrm d}{\mathrm d\mu}a_{ij}(\mu,z)z_{i}\zeta_{j}
\\&
=
2h^{-1}
\Bigl[
(\partial_ta_{ij})(\mu,z)z_{i}\zeta_{j}
+h^{-1} (\partial_ka_{ij})(\mu,z)a_{kl}(\mu,z)z_{i}\zeta_{j}\zeta_{l}
\\&\phantom{{}={}}{}
+h^{-1} a_{ij}(\mu,z)a_{ik}(\mu,z)\zeta_{j}\zeta_{k}
\\&\phantom{{}={}}{}
+a_{ij}(\mu,z)z_{i}\Bigl(-\tfrac12 h^{-1}(\partial_ja_{kl})(\mu,z)\zeta_{k}\zeta_{l}-h(\partial_jV)(\mu,z)
\Bigr)\Bigr]
.
\end{split}
\label{25062017}
\end{align}
Then, using the Cauchy--Schwarz inequality and \eqref{25061920},
and retaking $h_0\in (0,1]$ smaller if necessary, we obtain 
\begin{equation}
\tfrac{\mathrm d^2}{\mathrm d\mu^2}z^2
\ge h^{-2}(c_2-C_4\langle z\rangle^{-1-\epsilon})
.
\label{25062012}
\end{equation}

\smallskip
\noindent
\textit{Step 3.}\ 
Here, letting $h\in (0,1]$ be even smaller if necessary,
we deduce that for any $h\in (0,h_0]$ and $\mu\in [s,0]$
\begin{equation}
|z(\mu;s,y,\eta)|\ge c_3h^{-1}(\mu-s)-C_5
\label{2506192045}
.
\end{equation}
To prove this take $M\ge 0$ large enough that 
\begin{equation}
c_2-C_4\langle M\rangle^{-1-\epsilon}\ge c_4>0,
\quad \text{or}\quad 
\tfrac{\mathrm d^2}{\mathrm d\mu^2}z^2
\ge c_4h^{-2}\ \ \text{if }|z|\ge M
.
\label{25062013}
\end{equation}
Recalling the notation from \eqref{250620}, and then 
by the assumption we can find $t_0>s$ such that 
\[
|x(t_0;s,y,\eta)|\ge  M,\quad 
a_{ij}(t_0,x(t_0;s,y,\eta))x_i(t_0;s,y,\eta)\xi_j(t_0;s,y,\eta)>0. 
\]
On the other hand, write down the equations that 
\[
\bigl(z(s+h(t-s);s,y,\eta),\zeta(s+h(t-s);s,y,\eta)\bigr)
\]
satisfy, and compare them with \eqref{26030522},
and then by continuity of a solution to ODE in parameters 
for each $t\in\mathbb R$ we have 
\[
\lim_{h\to+0}
\bigl(z(s+h(t-s);s,y,\eta),\zeta(s+h(t-s);s,y,\eta)\bigr)
=
(x(t;s,y,\eta), \xi(t;s,y,\eta))
.\]
Therefore, if we take $h_0\in(0,1]$ smaller, we have for any $h\in (0,h_0]$
\begin{align}
\begin{split}
&|z(s+h(t_0-s);s,y,\eta)|\ge M,\\
&
a_{ij}(t_0,z(s+h(t_0-s);s,y,\eta))
\\&{}\cdot z_{i}(s+h(t_0-s);s,y,\eta)\zeta_{j}(s+h(t_0-s);s,y,\eta)>0
.
\end{split}
\label{25062014}
\end{align}
Then we learn that $|z(\mu;s,y,\eta)|$ is non-decreasing in $\mu\in [s+h(t_0-s),0]$,
since, otherwise, by \eqref{25062014} for some $\mu_0\in [s+h(t_0-s),0)$
\begin{equation}
\tfrac{\mathrm d}{\mathrm d\mu}z^2(\mu)>0\ \ \text{for }\mu \in(s+h(t_0-s),\mu_0]
,\quad 
\tfrac{\mathrm d}{\mathrm d\mu}z^2(\mu_0)=0
,
\label{25080319}
\end{equation}
however for that by \eqref{25062013} 
we have to have $|z|<M$ for some $\mu\in (s+h(t_0-s),\mu_0]$,
which contradicts \eqref{25062014} and \eqref{25080319}. Thus we have 
\[
\tfrac{\mathrm d^2}{\mathrm d\mu^2}z^2
\ge c_5h^{-2}\ \ \text{for }\mu\in [s+h(t_0-s),0],
\]
and this along with \eqref{25062014} implies 
\[
|z|^2
\ge M^2+c_6h^{-2}(\mu-s-h(t_0-s))^2\ \ \text{for }\mu\in [s+h(t_0-s),0]
.
\]
Since the orbit $\{(z(\mu;s,y,\eta),\zeta(\mu;s,y,\eta));\ 
\mu\in [s,s+h(t_0-s)]\}$ converges uniformly to 
$\{(x(t;s,y,\eta), \xi(t;s,y,\eta));\ t\in [s,t_0]\}$ as $h\to+0$,
we obtain the claim \eqref{2506192045}. 

\smallskip
\noindent
\textit{Step 4.}\ 
By \eqref{25062016}, \eqref{2506192045} and \eqref{25061920} we have 
\[
\bigl|\tfrac{\mathrm d}{\mathrm d\mu}\bigl(a_{ij}(\mu,z)\zeta_{i}\zeta_{j}\bigr)\bigr|
\le 
C_6\langle h^{-1}(\mu-s)\rangle^{-1-\epsilon}
a_{ij}(\mu,z)\zeta_{i}\zeta_{j}
+C_6h\langle h^{-1}(\mu-s)\rangle^{-\epsilon}
,
\]
from which it follows that for any $h\in (0,h_0]$ and $\mu\in [s,0]$ 
\begin{align}
\begin{split}
&\bigl(1-C_7h\bigr)a_{ij}(s,y)\eta_{i}\eta_{j}
-C_7h^{1+\epsilon}
\\&
\le a_{ij}(\mu,z)\zeta_{i}\zeta_{j}
\le 
\bigl(1+C_7h\bigr)a_{ij}(s,y)\eta_{i}\eta_{j}
+C_7h^{1+\epsilon}
\end{split}
\label{2506201715}
\end{align}
Then we combine \eqref{25062017}, \eqref{2506201715}, \eqref{2506192030} and \eqref{2506192045}
to deduce 
\begin{equation*}
\bigl|\tfrac{\mathrm d^2}{\mathrm d\mu^2}z^2
-2h^{-2}a_{ij}(s,y)\eta_{i}\eta_{j}
\bigr|
\le C_8h^{-2}\langle h^{-1}(\mu-s)\rangle^{-1-\epsilon}
,
\end{equation*}
so that 
\[
\bigl|z^2
-h^{-2}(\mu-s)^2a_{ij}(s,y)\eta_{i}\eta_{j}
\bigr|
\le C_9h^{-1}(\mu-s)+C_9
.
\]
Thus we obtain the assertion.
\end{proof}

\subsubsection{Time-independent Hamiltonian without potential}

Here we compare the orbit $(z(\mu),\zeta(\mu))$ with 
\[
\bigl(\widetilde x(\mu),\widetilde \xi(\mu)\bigr)
=
\bigl(\widetilde x(\mu;s,y,\eta),\widetilde \xi(\mu;s,y,\eta)\bigr)
:=\bigl(x(\mu;s,y,h^{-1}\eta),h\xi(\mu;s,y,h^{-1}\eta)\bigr)
, 
\]
see \eqref{250620} for the notation on the above right-hand side. Obviously, it satisfies 
\[
\tfrac{\mathrm d }{\mathrm d\mu}\widetilde x_{i}
=
h^{-1} a_{ij}(s,\widetilde x)\widetilde \xi_{j}
,\quad 
\tfrac{\mathrm d}{\mathrm d\mu}\widetilde \xi_{i}
=
-\tfrac12 h^{-1}(\partial_ia_{jk}(s,\widetilde x))\widetilde \xi_{j}\widetilde \xi_{k}
\]
with initial data $(\widetilde x(s),\widetilde \xi(s))=(y,\eta)$, 
and these are pretty close to \eqref{25061917},  \eqref{25061918} and \eqref{2506191919} 
satisfied by $(z(\mu),\zeta(\mu))$. 
Such comparison in short-time has already been done in the proof of Lemma~\ref{250804}, 
but here we do so globally. 

By the scaling argument the following limits are straightforward. 
Recall the notation $(x_\pm(s,y,\eta),\xi_\pm(s,y,\eta))$ from Proposition~\ref{250123}. 

\begin{lemma}\label{2603082235}
The following limits exist, and are given as 
\begin{align*}
&\lim_{(\mu-s)/h\to \pm\infty} \bigl(\widetilde x(\mu;s,y,\eta)-h^{-1}(\mu-s)\widetilde \xi(\mu;s,y,\eta)\bigr)
=x_\pm(s,y,\eta)
,
\\
&\lim_{(\mu-s)/h\to \pm\infty} \widetilde \xi(\mu;s,y,\eta)
=
\xi_\pm(s,y,\eta)
\end{align*}
locally uniformly in $(s,y,\eta)\in \Omega_\pm$ with $\pm s<0$, respectively. 
\end{lemma}
\begin{proof}
Note that by the scaling structure of the system \eqref{26030522} we have 
\begin{align}
\begin{split}
&\bigl(\widetilde x(\mu;s,y,\eta),\widetilde \xi(\mu;s,y,\eta)\bigr)
\\&
=\bigl(x(s+h^{-1}(\mu-s);s,y,\eta),\xi(s+h^{-1}(\mu-s);s,y,\eta)\bigr). 
\end{split}
\label{26031016}
\end{align}
Then the assertion follows immediately by Proposition~\ref{250123}. 
\end{proof}

Now we obtain a precise limiting behavior of the orbit $(z(\mu),\zeta(\mu))$.

\begin{lemma}\label{2603082236}
The following limits exist, and are given as 
\begin{align*}
&\lim_{h\to +0} 
\bigl\{
\bigl(z(\mu;s,y,\eta)-h^{-1}(\mu-s)\zeta(\mu;s,y,\eta)\bigr)
\\&\qquad\quad
-\bigl(\widetilde x(\mu;s,y,\eta)-h^{-1}(\mu-s)\widetilde \xi(\mu;s,y,\eta)\bigr)
\bigr\}
=0
,
\\
&\lim_{h\to +0} 
\bigl(\zeta(\mu;s,y,\eta)
-\widetilde \xi(\mu;s,y,\eta)\bigr)
=0
\end{align*}
locally uniformly in $(s,y,\eta)\in \Omega_\pm$ with $\pm s<0$ and $\pm\mu\in [\pm s,0]$, respectively. 
\end{lemma}
\begin{proof}
Let us discuss only the upper sign. The following estimates are all locally uniform in 
$(s,y,\eta)\in \Omega_+$ with $s<0$ and $\mu\in [s,0]$ unless otherwise mentioned.
The differences $z-\widetilde x$ and $\zeta-\widetilde\xi$ satisfy
\begin{align}
\tfrac{\mathrm d}{\mathrm d\mu}(z_{i}-\widetilde x_{i})
&=
h^{-1} a_{ij}(\mu,z)\zeta_{j}
-h^{-1} a_{ij}(s,\widetilde x)\widetilde \xi_{j}
,
\label{25062210}
\\
\begin{split}
\tfrac{\mathrm d}{\mathrm d\mu}(\zeta_{i}-\widetilde \xi_{i})
&=
-\tfrac12 h^{-1}(\partial_ia_{jk}(\mu,z))\zeta_{j}\zeta_{k}
+\tfrac12 h^{-1}\bigl(\partial_ia_{jk}(s,\widetilde x)\bigr)\widetilde \xi_{j}\widetilde \xi_{k}
\\&\phantom{{}={}}{}
-h(\partial_iV)(\mu,z)
.
\end{split}
\label{25062211}
\end{align}
By \eqref{25062211} and Lemma~\ref{250804} 
we have for any $h\in (0,h_0]$ and $\mu\in [s,0]$
\begin{align*}
\big|\tfrac{\mathrm d}{\mathrm d\mu}(\zeta-\widetilde \xi)\bigr|
&\le 
C_1h^{-1}\langle h^{-1} (\mu-s)\rangle^{-2-\epsilon} 
.
\end{align*}
Thus for any $M>0$ we have uniformly in $h\in (0,\min\{h_0,-s/M\}]$ and $\mu\in [s+hM,0]$  
\begin{align*}
&
\big|\zeta(\mu;s,y,\eta)-\widetilde \xi(\mu;s,y,\eta)\bigr|
\\&
\le 
\big|\zeta(s+hM;s,y,\eta)-\widetilde \xi(s+hM;s,y,\eta)\bigr|
+C_2h^{-1}\int_{s+hM}^{\mu}\langle h^{-1} (\nu-s)\rangle^{-2-\epsilon} \,\mathrm d\nu
\\&
\le 
\big|\zeta(s+hM;s,y,\eta)-\widetilde \xi(s+hM;s,y,\eta)\bigr|
+C_2\int_{M}^{\infty}\langle \nu\rangle^{-2-\epsilon} \,\mathrm d\nu
.
\end{align*}
The second term on the right-hand side can be arbitrarily small by taking $M>0$ large enough.
On the other hand, by continuity of a solution to ODE in parameters we have 
\[
\lim_{h\to +0}\big|\zeta(\mu;s,y,\eta)-\widetilde \xi(\mu;s,y,\eta)\bigr|=0
\]
uniformly in $\mu\in[s,s+hM]$ cf.\ the proof of Lemma~\ref{250804}. 
Therefore we obtain the latter limits of the assertion. 

As for the former limits, we use \eqref{25062210}, \eqref{25062211} and Lemma~\ref{250804} to have 
\begin{align*}
\bigl|
\tfrac{\mathrm d}{\mathrm d\mu}
\bigl\{
\bigl(z-h^{-1}(\mu-s)\zeta\bigr)
-\bigl(\widetilde x-h^{-1}(\mu-s)\widetilde\xi\bigr)
\bigr\}\bigr|
\le h^{-1}\langle h^{-1}(\mu-s)\rangle^{-1-\epsilon}
.
\end{align*}
Then, similarly to the above, we obtain the former limits of the assertion. 
\end{proof}

\subsection{Proof of the main proposition}\label{260303}

We close this section with the proof of Proposition~\ref{2603011619}.

\begin{proof}[Proof of Proposition~\ref{2603011619}]
\textit{1.}\ 
It suffices to show the assertion for $h=1$. Fix any $\kappa\in [0,1]$. 
By smoothness of a solution to ODE in the initial data 
$\Phi_1(\kappa)=\Phi(\kappa)$ is smooth mapping from $\mathcal U_1(\kappa)$ to the image. 
On the other hand, the converse $\Phi_1(\kappa)^{-1}$ is constructed 
by using the uniqueness of a solution to ODE as follows: 
Given any $(s,y,\sigma,\eta)$ in the image, solve the equations \eqref{25042320}--\eqref{25042323}
with initial data 
\[
(t(\kappa),x(\kappa),\tau(\kappa),\xi(\kappa))=(s,y,\sigma,\eta),  
\]
and then 
\[
\Phi_1(\kappa)^{-1}(s,y,\sigma,\eta)=(t(0),x(0),\tau(0),\xi(0))
.\]
The mapping $\Phi_1(\kappa)^{-1}$ is also smooth, and thus the assertion is verified. 

\smallskip
\noindent
\textit{2.}\ 
We shall not discuss \eqref{260305} as remarked in Remark~\ref{26030822}. 
The locally uniform limits \eqref{2603082229} follows by  
\eqref{2603082230}, \eqref{2603082231}--\eqref{2603082233}, \eqref{2603082234} 
and Lemmas~\ref{2603082236}  and \ref{2603082235}. 

\smallskip
\noindent
\textit{3.}\ 
The asserted estimates for the $t$ component are trivial by the identity \eqref{2603082230},
and that for the $\tau$ component reduces to those for the $x$ and $\xi$ components 
by the identities \eqref{2603082232} and \eqref{2603082234}. 
Thus it suffices to discuss the $x$ and $\xi$ components.  
In this proof let us denote these components by $x'$ and $\xi'$, respectively, 
not to be confused with \eqref{250620}. 
Then by definition they satisfy \eqref{25042321b} and \eqref{25042323b} along with \eqref{25042324b},
and we use these equations to deduce the desired estimates.  

If $\alpha=0$, we can show the assertion by the changes of variables 
\eqref{2603082231} and \eqref{2603082233},
Lemmas~\ref{2603082235} and \ref{2603082236}, 
and the scaling property \eqref{26031016}, see also the proof of Lemma~\ref{250804}. 
For $|\alpha|\ge 1$ we adopt the induction. 
Apply $\partial^\alpha=\partial^\alpha_{s,y,\sigma,\eta}$ to \eqref{25042321b} and \eqref{25042323b}, 
and then by the chain rule, or more rigorously Fa\`a di Bruno's formula, 
the induction hypothesis, and the lower bound from \eqref{2506201715} with \eqref{2603082233} 
it follows that 
\begin{align*}
\bigl|\tfrac{\mathrm d}{\mathrm d\kappa} \partial^\alpha x'\bigr|&
\le 
C_1 h^{-1}\langle h^{-1}\kappa\rangle^{-1-\epsilon}
\bigl(\langle h^{-1}\kappa\rangle^{-1}| \partial^\alpha x'|
+|\partial^\alpha\xi'|
+1\bigr)
,
\\
\bigl|\tfrac{\mathrm d}{\mathrm d\kappa}\partial^\alpha\xi'\bigr|
&\le 
C_1h^{-1}\langle h^{-1}\kappa\rangle^{-2-\epsilon}
\bigl( \langle h^{-1}\kappa\rangle^{-1}| \partial^\alpha x'|
+|\partial^\alpha\xi'|
+1\bigr)
.
\end{align*}
These imply by the Cauchy--Schwarz inequality that 
\begin{align*}
\bigl|\tfrac{\mathrm d}{\mathrm d\kappa}\bigl(|\partial^\alpha x'|^2+|\partial^\alpha\xi'|^2\bigr)\bigr|
\le C_2 h^{-1}\langle h^{-1}\kappa\rangle^{-1-\epsilon}\bigl(|\partial^\alpha x'|^2+|\partial^\alpha\xi'|^2+1\bigr)
,
\end{align*}
so that by Gr\"onwall's inequality 
\[
|\partial^\alpha x'|^2+|\partial^\alpha\xi'|^2
\le C_3
. 
\]
Hence we are done. 
\end{proof}

\section{Singularities in general dimensional space}\label{26012514}

In this section we prove Theorem~\ref{250208}. We split the proof into two steps. 
First, we will reduce it to a Heisenberg-type equation with respect to a certain technical Hamiltonian. 
Once a semiclassical solution to this equation is constructed, 
Theorem~\ref{250208} is an immediate consequence of it. 
For construction of a solution to the Heisenberg equation, 
the corresponding classical mechanics studied in Section~\ref{2601251432} plays an essential role. 
This strategy is affected by Nakamura's paper~\cite{MR2488342}, 
however the associated Hamiltonians are quite different. 

In Section~\ref{26030113} we will give the reduction procedure,
and in Section~\ref{2603013} we will semiclassically solve the Heisenberg equation.

\subsection{Reduction of the first main result}\label{26030113}

Here we present the main proposition of the section, and deduce Theorem~\ref{250208} from it. 
We are going to discuss the Heisenberg equation of the form
\begin{equation}
\mathbf D_L A(\kappa)
:=
\tfrac{\mathrm d}{\mathrm d\kappa}A(\kappa)+\mathrm i[L(\kappa),A(\kappa)]
=0
\ \ \text{for }\kappa\in[0,1], 
\label{25042517}
\end{equation}
where the associated technical Hamiltonian $L(\kappa)$ is given by 
\begin{align}
\begin{split}
L(\kappa)
&
=-t \mathrm e^{-\mathrm i\kappa tK}\Bigl(
\tfrac12 p_i \bigl\{a_{ij}((1-\kappa)t,x)-\delta_{ij}\bigr\}p_j
+V((1-\kappa)t,x)\Bigr)\mathrm e^{\mathrm i\kappa tK}
\\&
=-\tfrac12 t p_i \bigl\{a_{ij}((1-\kappa)t,x-\kappa tp_x)-\delta_{ij}\bigr\}p_j
-tV((1-\kappa)t,x-\kappa tp_x)
.
\end{split}
\label{2603014}
\end{align}
Note for the second identity of \eqref{2603014} we have used \eqref{25020716}. 
The main proposition of the section concerns a semiclassical solution to \eqref{25042517}. 
In order to give a precise statement we introduce a simple symbol class, cf.\ a textbook~\cite{MR1872698} by Martinez. 

\begin{definition}\label{26030116}
For a general dimension $n\in\mathbb N$ 
define $S_{n}$ as a class of all symbols 
$a=a_h\in C^\infty(\mathbb R^n)$, possibly dependent on $h\in (0,h_0]$ for some $h_0>0$, 
such that for any $\alpha\in\mathbb N_0^n$ there exists $C>0$ such that 
uniformly in $h\in (0,h_0]$ and $z\in\mathbb R^n$
\[
|\partial^\alpha a(z)|\le C.
\]
\end{definition}
\begin{remark}
We will implicitly use other standard symbol classes in the later arguments, 
but, for all the statements in the paper, solely the above class suffices. 
\end{remark}

Next, we present the main proposition of the section. 
For that recall the notation from Definition~\ref{26030116} and \eqref{2603015}.

\begin{proposition}\label{260301}
Let $(s,y,\eta)\in \Omega_\pm$ with $\pm s<0$, respectively.
Then there exists a neighborhood $W\subset \mathbb R^{2(1+d)}$ of $(s,y, -\tfrac12a_{ij}(s,y)\eta_i\eta_j,\eta)$ 
such that for any $b\in C^\infty_{\mathrm c}(W)$ 
one can construct $\widetilde b\in C^\infty([0,1];S_{2(1+d)})$ with the following properties. 
\begin{enumerate}
\item
If one denotes 
\begin{equation*}
\widetilde B(\kappa)=\widetilde b^{\mathrm W}(\kappa; t,x,h^2p_t,hp_x)\ \ \text{for }\kappa\in [0,1],
\end{equation*}
then it solves the Heisenberg equation \eqref{25042517} in the semiclassical sense that 
\begin{equation}
\mathbf D_L \widetilde B(\kappa)
=\mathcal O(h^\infty)
\ \ \text{uniformly in }\kappa\in[0,1]
\label{25062214}
\end{equation}
with initial value  
\[
\widetilde B(0)=b^{\mathrm W}(t,x,h^2p_t,hp_x)
.
\]

\item
In addition, uniformly in $\kappa\in[0,1]$
\[
\widetilde b(\kappa)-b\circ\Phi_h(\kappa)^{-1}=\mathcal O(h)\ \ \text{in }S_{2(1+d)},
\quad \ 
\mathop{\mathrm{supp}}\widetilde b(\kappa)\subset 
\Phi_h(\kappa)(\mathop{\mathrm{supp}} b)
.
\]
\end{enumerate}
\end{proposition}
\begin{remark}
The right-hand side of \eqref{25062214} reads as an operator 
with kernel in the Schwartz class $\mathcal S(\mathbb R^{2(1+d)})$, 
as the symbol of the left-hand side (with $h$ removed) is uniformly compactly supported 
in $\mathbb R^{2(1+d)}$ due to the second assertion, cf.\ Lemma~\ref{260311}. 
\end{remark}

We postpone the proof of Proposition~\ref{260301} to Section~\ref{2603013}.
Below let us prove Theorem~\ref{250208} by using Proposition~\ref{260301}. 

\begin{proof}[Deduction of Theorem~\ref{250208} from Proposition~\ref{260301}]
\textit{Necessity}.\ 
Under the setting of Theorem~\ref{250208}, 
assume $(s,x_\pm,-\tfrac12\xi_\pm^2,\xi_\pm)\notin\mathop{\mathrm{qh\text{-}WF}}\nolimits^2(u_K)$.
It suffices to show that for any 
$b\in S_{2(1+d)}$ supported sufficiently close to $(s,y, -\tfrac12a_{ij}(s,y)\eta_i\eta_j,\eta)$
\begin{equation}
\bigl\langle u,b^{\mathrm W}(t,x,h^2 p_t,hp_x)u\bigr\rangle_{L^2_{t,x}}=\mathcal O(h^\infty)
,
\label{2603011616}
\end{equation}
where $\langle\cdot,\cdot\rangle_{L^2_{t,x}}$ reads in an extended sense as the $\mathcal S'$-$\mathcal S$ pairing. 
Let us choose $\widetilde b$ and $\widetilde B$ as in Proposition~\ref{260301}, and set 
for any $\kappa\in [0,1]$
\begin{align*}
I(\kappa)
=
\bigl\langle \mathrm e^{-\mathrm i\kappa tK}U((1-\kappa)t)\phi,
\widetilde B(\kappa)
\mathrm e^{-\mathrm i\kappa tK}U((1-\kappa)t)\phi\bigr\rangle_{L^2_{t,x}}
.
\end{align*}
Then direct computations show
\begin{equation*}
\tfrac{\mathrm d}{\mathrm d\kappa}I(\kappa)
=
\bigl\langle \mathrm e^{-\mathrm i\kappa tK}U((1-\kappa)t)\phi,
\mathbf D_L \widetilde B(\kappa)
\mathrm e^{-\mathrm i\kappa tK}U((1-\kappa)t)\phi\bigr\rangle_{L^2_{t,x}}
, 
\end{equation*}
so that by the assertion 1 of Proposition~\ref{260301} we obtain 
\[
\bigl\langle u,b^{\mathrm W}(t,x,h^2 p_t,hp_x)u\bigr\rangle_{L^2_{t,x}}
=\langle u,\widetilde B(0)u\rangle_{L^2_{t,x}}
=\langle u_K,\widetilde B(1)u_K\rangle_{L^2_{t,x}}+\mathcal O(h^\infty)
.
\]
Hence \eqref{2603011616} follows by the assertion 2 of Proposition~\ref{260301}, 
Proposition~\ref{2603011619} and the assumption. 

\smallskip
\noindent
\textit{Sufficiency}.\ 
We next assume $(s,y, -\tfrac12a_{ij}(s,y)\eta_i\eta_j,\eta)\notin \mathop{\mathrm{qh\text{-}WF}}\nolimits^2(u)$. 
Let us find a neighborhood $W_K\subset \mathbb R^{2(1+d)}$ of $(s,x_\pm,-\tfrac12\xi_\pm^2,\xi_\pm)$ such that 
for any $c_K\in C^\infty_{\mathrm c}(W_K)$ 
\begin{equation}
\bigl\|c_K^{\mathrm W}(t,x,h^2 p_t,hp_x)u_K\bigr\|_{L^2_{t,x}}=\mathcal O(h^\infty)
.
\label{2603012054}
\end{equation}
Fix a neighborhood $W\subset \mathbb R^{2(1+d)}$ of 
$(s,y, -\tfrac12a_{ij}(s,y)\eta_i\eta_j,\eta)$ as in Proposition~\ref{260301}. 
By the assumption we may let it be small enough that for any $c\in C^\infty_{\mathrm c}(W)$ 
\begin{equation*}
\bigl\|c^{\mathrm W}(t,x,h^2 p_t,hp_x)u\bigr\|_{L^2_{t,x}}
=\mathcal O(h^\infty)
.
\end{equation*}
Now choose any real-valued $c_1\in C^\infty_{\mathrm c}(W)$ satisfying 
$c_1(s,y, -\tfrac12a_{ij}(s,y)\eta_i\eta_j,\eta)=1$, 
and then there exist neighborhoods $W_K\Subset W_{K,1}\Subset \mathbb R^{2(1+d)}$ of 
$(s,x_\pm,-\tfrac12\xi_\pm^2,\xi_\pm)$ such that 
\begin{equation}
c_{K,1}:=c_1\circ \Phi_h(1)^{-1}\ge \tfrac12 \ \ \text{ on }W_K 
,\quad 
c_{K,1}\ge \tfrac14 \ \ \text{ on }W_{K,1}
\label{26030121}
\end{equation}
uniformly in small $h\in (0,1]$. 
We are going to verify that this $W_K$ in fact satisfies the desired property \eqref{2603012054}. 
For that we first apply Proposition~\ref{260301} to the composite symbol $b_1=c_1\,\#^{\mathrm{W,qh}}c_1$, 
and let $\widetilde b_1$ be the associated symbol. 
Then, similarly to the necessity part, it follows that 
\begin{align}
\begin{split}
\langle u_K,\widetilde b_1^{\mathrm W}(1;t,x,h^2 p_t,hp_x)u_K\rangle_{L^2_{t,x}}
&=\bigl\|c_1^{\mathrm W}(t,x,h^2 p_t,hp_x)u\bigr\|_{L^2_{t,x}}^2
+\mathcal O(h^\infty)
\\&
=\mathcal O(h^\infty)
. 
\end{split}
\label{2603012130}
\end{align}
However, since $c_{K,1}\,\#^{\mathrm{W,qh}}c_{K,1}$ gives a principal part of $\widetilde b(1)$, 
\eqref{2603012130} implies that 
\begin{equation}
\bigl\|c_{K,1}^{\mathrm W}(t,x,h^2 p_t,hp_x)u_K\bigr\|_{L^2_{t,x}}^2=\mathcal O(h)
.
\label{2603012127}
\end{equation}
Thus by \eqref{2603012127}, \eqref{26030121} and the standard pseudodifferential calculus we can 
verify that for any $c_K\in C^\infty_{\mathrm c}(W_{K,1})$ 
\begin{equation}
\bigl\|c_K^{\mathrm W}(t,x,h^2 p_t,hp_x)u_K\bigr\|_{L^2_{t,x}}=\mathcal O(h^{1/2})
.
\label{26030120}
\end{equation}
We next choose $c_2\in C^\infty_{\mathrm c}(W)$ and $W_{K,2}\subset \mathbb R^{2(1+d)}$ with 
$W_K\Subset W_{K,2}\Subset W_{K,1}$ such that 
\begin{equation}
c_{K,2}:=c_2\circ \Phi_h(1)^{-1}\ge \tfrac14 \ \ \text{ on }W_{K,2}
,\quad 
\mathop{\mathrm{supp}}c_{K,2}\Subset W_{K,1}
\label{26030139}
\end{equation}
uniformly in small $h\in(0,1]$. We then repeat the above arguments. 
This time, due to the properties \eqref{26030139} and \eqref{26030120}, we obtain 
\begin{equation*}
\bigl\|c_{K,2}^{\mathrm W}(t,x,h^2 p_t,hp_x)u_K\bigr\|_{L^2_{t,x}}^2=\mathcal O(h^2)
, 
\end{equation*}
which along with \eqref{26030139} again implies that for any  $c_K\in C^\infty_{\mathrm c}(W_{K,2})$
\begin{equation*}
\bigl\|c_K^{\mathrm W}(t,x,h^2 p_t,hp_x)u_K\bigr\|_{L^2_{t,x}}=\mathcal O(h)
.
\end{equation*}
We thus inductively obtain \eqref{2603012054}. The proof is done. 
\end{proof}

\subsection{Construction of solution to Heisenberg equation}\label{2603013}

Here we prove Proposition~\ref{260301}. 
In our settings it is unclear how to apply the semiclassical pseudodifferential calculus to 
the equation \eqref{25062214}, since $L$ does not have semiclassical parameter,
and its symbol is in a rather bad class. 
We in fact split it into several parts, so that it fits into the semiclassical framework. 

Let 
\[
L(\kappa)=l^{\mathrm W}(\kappa;t,x,p_t,p_x)=l^{\mathrm W}(\kappa;t,x,p_x)
.
\]
For any $M>\epsilon>0$ 
we introduce cut-off functions $\chi_{\mathrm I}\in C^\infty_{\mathrm c}(\mathbb R)$, 
$\chi_{\mathrm{I\hspace{-.03em}I}},\chi_{\mathrm{I\hspace{-.03em}I}}'\in C^\infty_{\mathrm c}(\mathbb R^d\setminus\{0\})$ 
and $\chi_{\mathrm{I\hspace{-.03em}I\hspace{-.03em}I}}\in C^\infty_{\mathrm c}(\mathbb R^d)$ such that 
\begin{align*}
\chi_{\mathrm I}(t)&=1 \ \ \text{for } |t|<2M,
\\
\chi_{\mathrm{I\hspace{-.03em}I}}(\xi)&=\chi_{\mathrm{I\hspace{-.03em}I}}'(h\xi)=1 \ \ \text{for } \tfrac12\epsilon<h|\xi|<2M
,
\\
\chi_{\mathrm{I\hspace{-.03em}I\hspace{-.03em}I}}(x)&
=1 \ \ \text{for } |x|<2M, 
\end{align*}
and split 
\[
l
=
l_{\mathrm I} 
+l_{\mathrm{I\hspace{-.03em}I}}
+l_{\mathrm{I\hspace{-.03em}I\hspace{-.03em}I}} 
+l_{\mathrm{I\hspace{-.03em}V}} 
=\bar\chi_{\mathrm I} l
+\chi_{\mathrm I}\bar\chi_{\mathrm{I\hspace{-.03em}I}} l
+\chi_{\mathrm I}\chi_{\mathrm{I\hspace{-.03em}I}}\bar\chi_{\mathrm{I\hspace{-.03em}I\hspace{-.03em}I}} l
+\chi_{\mathrm I}\chi_{\mathrm{I\hspace{-.03em}I}}\chi_{\mathrm{I\hspace{-.03em}I\hspace{-.03em}I}} l 
,
\]
where $\bar\chi_*=1-\chi_*$ for 
$*=\mathrm I,\mathrm{I\hspace{-.05em}I}, \mathrm{I\hspace{-.05em}I\hspace{-.05em}I}$.

\begin{lemma}\label{260311}
Let $M>\epsilon>0$. 
\begin{enumerate}
\item
For any $*=\mathrm I,\mathrm{I\hspace{-.05em}I}, \mathrm{I\hspace{-.05em}I\hspace{-.05em}I}$
and $b\in S_{2(1+d)}$ with 
\begin{equation}
\mathop{\mathrm{supp}}b
\subset \bigl\{(t,x,\tau,\xi);\ |t|<M,\ |x|<M,\ \epsilon<|\xi|<M\bigr\}
\label{260310}
\end{equation}
one has 
\begin{align*}
l_*^{\mathrm W}(\kappa;t,x,p_x)\circ b^{\mathrm W}(t,x,h^2p_t,hp_x)&=\mathcal O(h^\infty),\\ 
b^{\mathrm W}(t,x,h^2p_t,hp_x)\circ l_*^{\mathrm W}(\kappa;t,x,p_x)&=\mathcal O(h^\infty),
\end{align*}
as operators with kernels in the Schwartz class, uniformly in $\kappa\in[0,1]$. 
\item
The symbol $l_{\mathrm{I\hspace{-.03em}V}} $ satisfies that 
for any $k\in\mathbb N_0$ and $\alpha,\beta \in \mathbb N_0^d$ there exists $C>0$ such that 
uniformly in $\kappa\in [0,1]$ and $(t,x,\xi)\in\mathbb R^{1+2d}$
\begin{align*}
&\bigl|\partial_t^k\partial_x^\alpha\partial_\xi^\beta l_{\mathrm{I\hspace{-.03em}V}} (\kappa;t,x,\xi)\bigr|\le Ch^{-2-k+|\beta|},
\\
&\bigl|\partial_t^k\partial_x^\alpha\partial_\xi^\beta (l_{\mathrm{I\hspace{-.03em}V}} -l_{0,\mathrm{I\hspace{-.03em}V}} )(\kappa;t,x,\xi)\bigr|\le Ch^{-k+|\beta|}
,
\end{align*}
where $l_{0,\mathrm{I\hspace{-.03em}V}}=\chi_{\mathrm I}\chi_{\mathrm{I\hspace{-.03em}I}}
\chi_{\mathrm{I\hspace{-.03em}I\hspace{-.03em}I}} l_0$ 
with $l_0$ from \eqref{250619}. 
\end{enumerate}
\end{lemma}
\begin{remark}
The symbol $ l_{\mathrm{I\hspace{-.03em}V}}(\kappa)$ remains to be in a very bad class with respect to 
$t$-derivatives, but it does not cause a big problem. 
This is because we always compose $ l_{\mathrm{I\hspace{-.03em}V}}(\kappa)$ 
with symbols having $h^2$ in front of $\tau$ variable,
and also because $\tau$ is absent from $ l_{\mathrm{I\hspace{-.03em}V}}(\kappa)$, 
and we will not need to take an oscillatory integral, or integrate it by parts, in $t$.  
\end{remark}
\begin{proof}
\textit{1}.\ 
We prove only the first estimate since the second can be treated similarly. 
We can see that the associated operator kernel of the left-hand side is given by 
\begin{align*}
K_*(\kappa;t,x,s,z)
&=(2\pi)^{-(1+2d)}\int_{\mathbb R^{1+3d}}
\mathrm e^{\mathrm i(x-y)\xi+\mathrm i(t-s)\tau+\mathrm i(y-z)\eta}
\\&\qquad \qquad\qquad\qquad {}\cdot
l_*\bigl(\kappa;t,\tfrac{x+y}2,\xi\bigr)b\bigl(\tfrac{t+s}2,\tfrac{y+z}2,h^2\tau,h\eta\bigr)
\,\mathrm d\tau\mathrm d\eta\mathrm dy\mathrm d\xi
. 
\end{align*}
Then the desired estimate follows by support property of the integrand and repeated integrations by parts,
cf.\ \eqref{250619}. 
These are rather standard arguments in the microlocal analysis, 
and we may omit the details. 

\smallskip
\noindent
\textit{2}.\ 
The asserted estimates follow by 
using the support property of the cut-off function
$\chi_{\mathrm I}\chi_{\mathrm{I\hspace{-.03em}I}}\chi_{\mathrm{I\hspace{-.03em}I\hspace{-.03em}I}}$ 
and the asymptotic expansion formula for a composition of the Weyl quantization. 
Let us again omit the details. Thus we are done. 
\end{proof}

We are ready to prove Proposition~\ref{260301}.

\begin{proof}[Proof of Proposition~\ref{260301}]
Let $b\in C^\infty_{\mathrm c}(W)$ be as in the assertion with $W$ being sufficiently small. 
We are going to construct a symbol $\widetilde b(\kappa)$ as an asymptotic sum 
\[
\widetilde b(\kappa)\sim \sum_{j\in\mathbb N_0} h^j\widetilde b_j(\kappa)
.
\]

First, we set 
\begin{equation}
\widetilde b_0(\kappa)=b\circ\Phi_h(\kappa)^{-1}\in S_{2(1+d)}
.
\label{25062222}
\end{equation}
Since $\mathop{\mathrm{supp}}\widetilde b(\kappa)=\Phi_h(\kappa)(\mathop{\mathrm{supp}} b)$ 
uniformly satisfies \eqref{260310} for some $M>\epsilon>0$, 
the assertion 1 from Proposition~\ref{260301} applies to it, and we have 
\begin{equation}
\mathbf D_L\widetilde b_0(\kappa;t,x,h^2p_t,hp_x)
=\mathbf D_{L_{\mathrm{I\hspace{-.03em}V}}}\widetilde b_0(\kappa;t,x,h^2p_t,hp_x)+\mathcal O(h^\infty) 
,
\label{2603108}
\end{equation}
where $L_{\mathrm{I\hspace{-.03em}V}}(\kappa)=l_{\mathrm{I\hspace{-.03em}V}}^{\mathrm W}(\kappa;t,x,p_x)$. 
Then by the assertion 2 from Proposition~\ref{260301} 
the standard symbol calculus is available to compute a principal symbol of 
the first term on the right-hand side of \eqref{2603108}. 
In fact, it vanishes since $\widetilde b_0$ satisfies by its definition \eqref{25062222} that 
\begin{align}
\begin{split}
0&=(\partial_\kappa b_0)(\kappa;t,x,h^2\tau,h\xi)
\\&\phantom{{}={}}{}
+(\partial_\tau l_0)(\kappa;t,x,\xi)(\partial_t b_0)(\kappa;t,x,h^2\tau,h\xi)
\\&\phantom{{}={}}{}
+(\partial_\xi l_0)(\kappa;t,x,\xi)(\partial_xb_0)(\kappa;t,x,h^2\tau,h\xi)
\\&\phantom{{}={}}{}
-h^2(\partial_tl_0)(\kappa;t,x,\xi)(\partial_\tau b_0)(\kappa;t,x,h^2\tau,h\xi)
\\&\phantom{{}={}}{}
-h(\partial_xl_0)(\kappa;t,x,\xi)(\partial_\xi b_0)(\kappa;t,x,h^2\tau,h\xi)
,
\end{split}
\label{25062221}
\end{align}
where $l_0$ is given by \eqref{250619}. 
Therefore, also noting that a remainder term of commutator of 
the Weyl quantization has one more extra decay in $h$, 
we can find $r_0\in C^\infty([0,1];S_{2(1+d)})$ such that 
\begin{equation*}
\mathbf D_L\widetilde b_0(\kappa;t,x,h^2p_t,hp_x)
=
hr_0(\kappa;t,x,h^2p_t,hp_x)
+\mathcal O(h^\infty) 
,
\end{equation*}
and that 
\[
\mathop{\mathrm{supp}}r_0(\kappa)\subset 
\Phi_h(\kappa)(\mathop{\mathrm{supp}} b)
. 
\]

We next let 
\begin{equation*}
\widetilde b_1(\kappa)=-\int_0^\kappa r_0\circ\Phi_h(\mu)\circ\Phi_h(\kappa)^{-1}\,\mathrm d\mu
\in S_{2(1+d)}
.
\end{equation*}
Then, similarly to $\widetilde b_0(\kappa)$, we can find  $r_0\in C^\infty([0,1];S_{2(1+d)})$ such that 
\begin{equation*}
\mathbf D_L\widetilde b_1(\kappa;t,x,h^2p_t,hp_x)
=
h^2r_1(\kappa;t,x,h^2p_t,hp_x)
+\mathcal O(h^\infty) 
,
\end{equation*}
and that 
\[
\mathop{\mathrm{supp}}r_1(\kappa)\subset 
\Phi_h(\kappa)(\mathop{\mathrm{supp}} b)
. 
\]

Repeating the above arguments, we can construct $\widetilde b_j(\kappa)$ for all $j\in\mathbb N_0$,
and thus we are done. 
\end{proof}

\section{Singularities in one-dimensional space}\label{26012420}

In this section we prove Theorem~\ref{250208b}. 
Let us outline the proof. 
A natural idea would be to compare supports of two symbols 
defining the relevant wave front sets of $u_K$ and $\phi$.
However, they have different numbers of arguments, and are not directly comparable. 
Therefore, we first convert them to be defined on the same space to be comparable.  
Then it turns out that the converted symbol concerning $\phi$ has 
much wider support than that concerning $u_K$. 
This implies that the former assertion of Theorem~\ref{250208b} is trivial
whereas the latter would be hopeless. 
However, in the one-dimensional case, we can somehow recover it by using a special partition of unity, 
and the explicit formula \eqref{25020716} involving the free propagator. 

We present these preliminaries in Section~\ref{260201b},
and implement the proof in Section~\ref{260201}.

\subsection{Preliminaries}\label{260201b}

\subsubsection{Symbol conversion}

Let us first rewrite the two conditions from Theorem~\ref{250208b},
employing symbols living in the same dimensional space. 

\begin{lemma}\label{260107}
Let $\phi\in \mathcal S'(\mathbb R^d)$, 
and let $u_K\in \mathcal S'(\mathbb R^{1+d})$ be from \eqref{250204}. 
\begin{enumerate}
\item\label{260107b}
Let $(s,y,\eta)\in \mathbb R^{1+d}\times(\mathbb R^d\setminus\{0\})$.
One has $(s,y,-\tfrac12\eta^2,\eta)\notin \mathop{\mathrm{qh\text{-}WF}}^2(u_K)$ 
if and only if there exists a neighborhood $U\subset \mathbb R^{1+2d}$ of $(s,y,\eta)$ such that 
for any $a\in S_{1+2d}$ supported in  $U$ 
\begin{equation}
\|a^{\mathrm W}(t,x,hp_x)u_K\|_{L^2_{t,x}}=\mathcal O(h^\infty)
\label{26021816}
\end{equation}
as $h\to +0$.
\item\label{260120}
Let $(s,y,\eta)\in \mathbb R\times (\mathbb R^{2d}\setminus\{0\})$. 
One has $(y,\eta)\not\in \mathop{\mathrm{HWF}}(\phi)$
if and only if there exists a neighborhood $U\subset \mathbb R^{1+2d}$ of $(s,y+s\eta,\eta)$
such that for any $a\in S_{1+2d}$ supported in $U$  
\begin{equation}
\|a^{\mathrm W}(t,hx,hp_x)u_K\|_{L^2_{t,x}}=\mathcal O(h^\infty)
\label{26021817}
\end{equation}
as $h\to +0$.
\end{enumerate}
\end{lemma}
\begin{remark}\label{2602181606}
If we rewrite the two conditions of Theorem~\ref{250208b} by using Lemma~\ref{260107}, 
we can see that the pseudodifferential operator corresponding to \eqref{26021817} has a much wider microlocal support 
than the one to \eqref{26021816}. 
Thus the former assertion of Theorem~\ref{250208b} is trivial. 
\end{remark}

\begin{proof} 
\textit{\ref{260107b}}.\ 
The sufficiency is obvious, and it suffices to show the necessity. 
Assume $(s,y,-\tfrac12\eta^2,\eta)\notin \mathop{\mathrm{qh\text{-}WF}}^2(u_K)$. 
By the standard arguments in the pseudodifferential calculus 
we can find a neighborhood $V\subset \mathbb R^{2(1+d)}$ of $(s,y,-\tfrac12\eta^2,\eta)$
such that for any $b\in S_{2(1+d)}$ supported in $V$ 
\begin{equation}
\|b^{\mathrm W}(t,x,h^2p_t,hp_x)u_K\|_{L^2_{t,x}}=\mathcal O(h^\infty)
.
\label{26020616}
\end{equation}
Then we choose a small neighborhood $U\subset \mathbb R^{1+2d}$ of $(s,y,\eta)$ such that 
\[
S=\bigl\{\bigl(t,x,-\tfrac12\xi^2,\xi\bigr)\in\mathbb R^{2(1+d)};\ (t,x,\xi)\in U\bigr\}
\Subset V
,
\]
and fix $\chi\in C^\infty_{\mathrm c}(V)$ such that $\chi=1$ in a neighborhood of $S$.
Now we take any $a\in S_{1+2d}$ supported in $U$, and split 
\begin{align}
\begin{split}
a^{\mathrm W}(t,x,hp_x)u_K(t,x)
&=
(\chi a)^{\mathrm W}(t,x,h^2p_t,hp_x)u_K(t,x)
\\&\phantom{{}={}}{}
+((1-\chi) a)^{\mathrm W}(t,x,h^2p_t,hp_x)u_K(t,x)
.
\end{split}
\label{2602061647}
\end{align}
Then the first term on the right-hand side of \eqref{2602061647} is $\mathcal O(h^\infty)$ 
due to \eqref{26020616}, and hence it suffices to discuss the second. 
Note that we have an integral expression
\begin{align*}
&
(2\pi)^{1+d}((1-\chi) a)^{\mathrm W}(t,x,h^2p_t,hp_x)u_K(t,x)
\\&
=\int_{\mathbb R^{2(1+d)}}
\mathrm e^{\mathrm i(t-r)\tau+\mathrm i(x-z)\xi}
((1-\chi) a)\bigl(\tfrac12(t+r),\tfrac12(x+z),h^2\tau,h\xi\bigr)
u_K(r,z)
\,\mathrm dr\mathrm dz\mathrm d\tau\mathrm d\xi 
.
\end{align*}
We first integrate it by parts by using the identity 
\[
\mathrm e^{\mathrm i(t-r)\tau+\mathrm i(x-z)\xi}
=(2\tau+\xi^2)^{-1}
(-2p_r+p_z^2)\mathrm e^{\mathrm i(t-r)\tau+\mathrm i(x-z)\xi}
,
\]
which is non-singular on $\mathop{\mathrm{supp}}((1-\chi)a)$. 
Then we can write it as 
\begin{align*}
&
(2\pi)^{1+d}((1-\chi) a)^{\mathrm W}(t,x,h^2p_t,hp_x)u_K(t,x)
\\&
=\int_{\mathbb R^{2(1+d)}}
\mathrm e^{\mathrm i(t-r)\tau+\mathrm i(x-z)\xi}
(2\tau+\xi^2)^{-1}
\\&\qquad \qquad \quad {}
\cdot 
(2p_r+p_z^2)\bigl[((1-\chi) a)\bigl(\tfrac12(t+r),\tfrac12(x+z),h^2\tau,h\xi\bigr)
u_K(r,z)\bigr]
\,\mathrm dr\mathrm dz\mathrm d\tau\mathrm d\xi
.
\end{align*}
Next, we use the product rule and the free Schr\"odinger equation $(2p_t+p_x^2)u_K=0$,
and then only $u_K$ and $p_zu_K$ are left in the integrand. 
However, we can replace $p_zu_K$ by $u_K$ by using integration by parts, so that we obtain 
\begin{align*}
&
(2\pi)^{1+d}((1-\chi) a)^{\mathrm W}(t,x,h^2p_t,hp_x)u_K(t,x)
\\&
=\int_{\mathbb R^{2(1+d)}}
\mathrm e^{\mathrm i(t-r)\tau+\mathrm i(x-z)\xi}
(L(1-\chi) a)\bigl(\tfrac12(t+r),\tfrac12(x+z),h^2\tau,h\xi\bigr)u_K(r,z)
\,\mathrm dr\mathrm dz\mathrm d\tau\mathrm d\xi
\end{align*}
with 
\[
L=\tfrac14(2\tau+\xi^2)^{-1}(4p_t+4\xi p_x-p_x^2)
. 
\]
We can repeat this procedure as many times as we want, and thus obtain 
\[
((1-\chi) a)^{\mathrm W}(t,x,h^2p_t,hp_x)u_K(t,x)
=
(2\pi)^{-1-d}\int_{\mathbb R^{1+d}}\mathcal K(t,x,r,z)u_K(r,z)\,\mathrm dr\mathrm dz
\]
with 
\[
\mathcal K(t,x,r,z)=\int_{\mathbb R^{2(1+d)}}
\mathrm e^{\mathrm i(t-r)\tau+\mathrm i(x-z)\xi}
(L^k(1-\chi) a)\bigl(\tfrac12(t+r),\tfrac12(x+z),h^2\tau,h\xi\bigr)
\,\mathrm d\tau\mathrm d\xi
\]
for any $k\in\mathbb N_0$. 
Obviously, the kernel $\mathcal K$ gives a smoothing operator of order $\mathcal O(h^\infty)$,
and thus we are done with the assertion~\ref{260107b}. 

\smallskip
\noindent
\textit{\ref{260120}}.\ 
We first assume that there exists a neighborhood $U\subset \mathbb R^{1+2d}$ of $(s,y+s\eta,\eta)$ as in the assertion. 
We are going to show that there exists $\epsilon>0$ such that 
for any $b\in S_{2d}$ supported in an $\epsilon$-neighborhood of $(y,\eta)$
\begin{equation}
\|b^{\mathrm W}(hx,hp_x)\phi\|_{L^2_{x}}
=\mathcal O(h^\infty)
\ \ \text{as }h\to +0
.
\label{260119}
\end{equation}
Fix $\chi\in C^\infty_{\mathrm c}([0,\infty))$ such that 
\begin{equation}
\chi(\lambda)=\begin{cases}1&\text{for }\lambda\in [0,1],\\0&\text{for }\lambda\in [2,\infty),\end{cases}
\quad
\chi'\le 0, 
\label{260121}
\end{equation}
and set for $\delta>0$
\[
\chi_\delta=\chi(\delta^{-1}|\cdot|).
\]
Then by \eqref{25020716} we can rewrite the left-hand side of \eqref{260119} as 
\begin{align}
\begin{split}
\|b^{\mathrm W}(hx,hp_x)\phi\|_{L^2_{x}}^2
&
=\|\chi_\delta\|_{L^2}^{-2}\int_{\mathbb  R}\chi_\delta(t-s)^2
\bigl\|b^{\mathrm W}(hx-thp_x,hp_x)\mathrm e^{-\mathrm itK}\phi\bigr\|_{L^2_{x}}^2
\,\mathrm dt
\\&
=
\|\chi_\delta\|_{L^2}^{-2}\|a^{\mathrm W}(t,hx,hp_x)u_K\|_{L^2_{t,x}}^2
\end{split}
\label{26021815}
\end{align}
with 
\begin{equation}
a(t,x,\xi)
=\chi_\delta(t-s)b(x-t\xi,\xi).
\label{260218}
\end{equation}
If we let $\epsilon,\delta>0$ be small enough, 
$a\in S_{1+2d}$ is supported in $U$, so that \eqref{260119} holds. 
Thus the sufficiency part of the second assertion is done. 

We can prove the necessity part similarly to the above by using a symbol of the form \eqref{260218}. 
It is straightforward, and we omit the details. We are done. 
\end{proof}

\subsubsection{Partition of unity generated by free classical flow}

For the latter assertion of Theorem~\ref{250208b}
we have to recover the stronger estimate \eqref{26021817} 
from the weaker one \eqref{26021816}, as remarked in Remark~\ref{2602181606}. 
The following technical partition of unity is a key,  
as well as the exact identity \eqref{25020716}. 

\begin{lemma}\label{260219}
Fix any $\epsilon>0$. 
Then there exists a sequence $(\chi_{m,n})_{(m,n)\in\mathbb Z^2}$
of functions belonging to $C^\infty_{\mathrm c}(\mathbb R^2)$ 
such that the following holds.  
\begin{enumerate}
\item
For any $(m,n)\in\mathbb Z^2$, $\mathop{\mathrm{supp}}\chi_{m,n}\subset [-2\epsilon,2\epsilon]\times 
[2^{n-2},2^{n+1}]$.
\item
For any $(k,l)\in\mathbb N_0^2$ there exists $C>0$ such that 
for any $(m,n)\in\mathbb Z^2$ and $(\mu,\nu)\in\mathbb R^2$
\[
|\partial_\mu^k\partial_\nu^l\chi_{m,n}(\mu,\nu)|\le C2^{-nl}
.
\]
\item
For any $(\mu,\nu)\in\mathbb R\times (0,\infty)$
\[
\sum_{(m,n)\in\mathbb Z^2}\chi_{m,n}(\mu-\epsilon m2^{-n}\nu,\nu)=1
.
\]
\end{enumerate}
\end{lemma}
\begin{proof}
Fix $\chi\in C^\infty_{\mathrm c}([0,\infty))$ satisfying \eqref{260121}, 
and we set 
\[
\widetilde \eta(\mu,\nu)=\chi(|\mu|/\epsilon)\bigl(\chi(\nu/2)-\chi(2\nu)\bigr)
.
\]
Then, since 
\begin{equation}
\widetilde \eta=1\ \ \text{on }[-\epsilon,\epsilon]\times [1,2]
,\quad 
\mathop{\mathrm{supp}}\widetilde\eta\subset [-2\epsilon,2\epsilon]\times [1/2,4]
, 
\label{2602182244}
\end{equation}
a sum 
\[
Y(\mu,\nu)=\sum_{(m,n)\in\mathbb Z^2}\widetilde \eta(\mu-\epsilon m2^{-n}\nu,2^{-n}\nu)
\]
is uniformly locally finite and uniformly positive on $\mathbb R\times (0,\infty)$. 
Now we define 
\begin{equation}
\chi_{m,n}(\mu,\nu)=Y\bigl(\mu+\epsilon m2^{-n}\nu,\nu\bigr)^{-1} \widetilde\eta(\mu,2^{-n}\nu)
,
\label{26021822}
\end{equation}
and verify the asserted properties. 
The properties 1 and 3 are clear from their construction. 
To verify the property 2 we shall count, on the right-hand side of \eqref{26021822}, 
how many summands in $Y(\mu+\epsilon m2^{-n}\nu,\nu)$ could survive,
or could have supports intersecting with that of $\widetilde\eta(\mu,2^{-n}\nu)$. 
Noting \eqref{2602182244}, we in fact have the expression \eqref{26021822} reduced at most to 
\begin{align*}
\chi_{m,n}(\mu,\nu)
&=
\Biggl(\sum_{j=n-2}^{n+2}\sum_{|i-m2^{j-n}|\le 100}
\widetilde \chi_{j}\bigl(\mu+\epsilon m2^{-n}\nu-\epsilon i2^{-j}\nu,2^{-j}\nu\bigr)
\Biggr)^{-1}\widetilde\eta(\mu,2^{-n}\nu)
.
\end{align*}
Thus the assertion follows. 
\end{proof}

\subsection{Proof of the second main result}\label{260201}

Now we are ready to prove Theorem~\ref{250208b}. 

\begin{proof}[Proof of Theorem~\ref{250208b}]
We have only to prove the latter assertion, see Remark~\ref{2602181606}.  
Assume $(s,y,-\tfrac12\eta^2,\eta)\notin \mathop{\mathrm{qh\text{-}WF}}^2(u_K)$.
Take a neighborhood $U\subset\mathbb R^{1+2d}$ of $(s,y,\eta)$ 
as in the first assertion of Lemma~\ref{260107},
and for small $\epsilon>0$ let $(\chi_{m,n})_{(m,n)\in\mathbb Z^2}$ 
be a sequence from Lemma~\ref{260219}. 
Let $a\in S_{1+2d}$ be supported sufficiently close to $(s,y,\eta)$, 
and we are going to verify the condition \eqref{26021817}. 
If we set 
\begin{align}
\begin{split}
a_{m,n}(t,x,h\xi)
&=\chi_{m,n}\bigl((x-y)\xi/|\xi|-\epsilon m 2^{-n}|\xi|,|\xi|\bigr)a(t,hx,h\xi)
,
\end{split}
\label{26021915}
\end{align}
then we have 
\[
a(t,hx,h\xi)
=\sum_{(m,n)\in\mathbb Z^2}a_{m,n}(t,x,h\xi)
=\sum_{|m|\le \delta/h,\, |n+\log_2 h|\le C}
a_{m,n}(t,x,h\xi)
\]
for some $\delta,C>0$. 
Here for the above second identity we have used the support properties of $\chi_{m,n}$. 
We note that, for any small $\epsilon>0$ fixed, 
by squeezing the support of $a$ small enough, 
we may let $\delta>0$ be arbitrarily small uniformly in $h\in (0,h_0]$ with some $h_0\ll 1$. 
By the above decomposition and \eqref{25020716} 
we can estimate the left-hand side of \eqref{26021817} as 
\begin{align*}
&
\bigl\|
a^{\mathrm W}(t,hx,hp_x)u_K
\bigr\|_{L^2_{t,x}}
\\&
\le 
\sum_{|m|\le \delta/h,\, |n+\log_2 h|\le C}
\bigl\|a_{m,n}^{\mathrm W}(t,x,hp_x)u_K(t,x)\bigr\|_{L^2_{t,x}}
\\&
\le 
\sum_{|m|\le \delta/h,\, |n+\log_2 h|\le C}
\bigl\|a_{m,n}^{\mathrm W}(t,x+\epsilon m 2^{-n}p_x,hp_x)u_K(t-\epsilon m 2^{-n},x)\bigr\|_{L^2_{t,x}}
\\&
\le 
\sum_{|m|\le \delta/h,\, |n+\log_2 h|\le C}
\bigl\|a_{m,n}^{\mathrm W}(t+\epsilon m 2^{-n},x+\epsilon m 2^{-n}p_x,hp_x)u_K(t,x)\bigr\|_{L^2_{t,x}}
.
\end{align*}
Let us take a look at the (semiclassical) symbols of the last pseudodifferential operators: 
\begin{align}
\begin{split}
b_{m,n}(t,x,\xi)
:={}&
a_{m,n}(t+\epsilon m 2^{-n},x+\epsilon m 2^{-n}h^{-1}\xi,\xi)
\\
={}&\chi_{m,n}\bigl((x-y)\xi/|\xi|,h^{-1}|\xi|\bigr)a\bigl(t+\epsilon m 2^{-n},hx+\epsilon m 2^{-n}\xi,\xi\bigr)
.
\end{split}
\label{26021916}
\end{align}
As noted above, first let $\epsilon>0$ be small, and then squeeze $\mathop{\mathrm{supp}}a$ small,
and we have $b_{m,n}$ supported in $U$ with their derivatives bounded uniformly in 
$|m|\le \delta/h$, $|n+\log_2 h|\le C$ and $h\in(0,h_0]$. 
This implies for any $N\in\mathbb N$
\[
\bigl\|b_{m,n}^{\mathrm W}(t,x,hp_x)u_K\bigr\|_{L^2_{t,x}}
\le C_Nh^N
\]
uniformly in $|m|\le \delta/h$, $|n+\log_2 h|\le C$ and $h\in(0,h_0]$, and hence 
\begin{equation*}
\bigl\|a^{\mathrm W}(t,hx,hp_x)u_K\bigr\|_{L^2_{t,x}}
\le C'_Nh^{N-1}. 
\end{equation*}
We are done. 
\end{proof}

\begin{remark}
The symbol $a_{m,n}$ from \eqref{26021915}, with $\xi$ replaced by $h^{-1}\xi$, does not belong to 
$S_{1+2d}$ uniformly in $|m|\le \delta/h$ and $|n+\log_2 h|\le C$
since the first factor on the right-hand side dissatisfies the required estimates. 
On the other hand, the transformed symbol $b_{m,n}$ from \eqref{26021916} does so 
due to the fact that the second factor stays in the same symbol class 
under conjugation by the free propagator. 
\end{remark}

\subsubsection*{Acknowledgements} 
KI was partially supported by JSPS KAKENHI Grant Number JP23K03163. 
The authors thank Akitoshi Hoshiya for some references.

\providecommand{\bysame}{\leavevmode\hbox to3em{\hrulefill}\thinspace}
\providecommand{\MR}{\relax\ifhmode\unskip\space\fi MR }
\providecommand{\MRhref}[2]{%
  \href{http://www.ams.org/mathscinet-getitem?mr=#1}{#2}
}
\providecommand{\href}[2]{#2}

\end{document}